\UseRawInputEncoding
\documentclass[twoside]{article}
\usepackage{amsmath,amssymb,amsthm,latexsym,bm,indentfirst,wasysym}
\usepackage{graphicx,epsfig,amscd,mathrsfs,multirow,bm}
\usepackage{cite}
\usepackage{caption}
\usepackage{color}
\usepackage[colorlinks,linkcolor=red,anchorcolor=red,citecolor=blue]{hyperref}
\numberwithin{equation}{section}
\usepackage{titlesec}
\titleformat{\section}{\large\bfseries}{\thesection}{1em}{}

\input xy
\xyoption{all}

\allowdisplaybreaks[4]

\newtheorem{theorem}{Theorem}[section]

\newtheorem{definition}[theorem]{Definition}
\newtheorem{proposition}[theorem]{Proposition}

\newtheorem{example}[theorem]{Example}
\newtheorem{remark}[theorem]{Remark}

\catcode`@=11
\long\def\@makefntext#1{\noindent #1}
\newskip\tabcentering \tabcentering=1000pt plus 1000pt minus 1000pt
\def\MCH#1#2{\setbox0=\hbox{\raise#1\hbox{#2}}\smash{\box0}}

\def\@evenfoot{}\def\@oddfoot{}

\def\@evenhead{\hbox to\textwidth{\small\rm\thepage \hfill
{\it  }}} 

\def\@oddhead{\hbox to \textwidth{\small{\it

} \hfill\thepage}}

\catcode`@=12

\def\bc{\begin{center}}
\def\ec{\end{center}}
\def\no{\noindent}

\begin{document}


\abovedisplayskip=6pt plus 1pt minus 1pt \belowdisplayskip=6pt
plus 1pt minus 1pt
\thispagestyle{empty} \vspace*{-1.0truecm} \noindent
\vskip 10mm \bc{\Large\bf Some estimates of weighted Hardy--Littlewood operators and their commutators on mixed Morrey-type spaces in the Dunkl setting
\footnotetext{\footnotesize
Supported by the National Natural Science Foundation of China (No. 12501128) and Higher Education Innovation Fund of Gansu Provincial Department of Education (Grant No.\,2026CXZX-422).\\
* Corresponding author\\
E-mail address: rhliu@nwnu.edu.cn (Liu Ronghui); zhangqimaths@126.com (Zhang Qi)
} } \ec

\vskip 5mm
\bc{Liu Ronghui$^{1}$,\ \ \ Zhang Qi$^{1,*}$}\\
{\small\it $1$.College of Mathematics and Statistics, Northwest Normal University, Lanzhou $730070$, Gansu, P. R. China.\\
}\ec

\vskip 1 mm

{\narrower\noindent{\small {\small\bf Abstract}\ \ In this paper, we consider the weighted Hardy--Littlewood operator
$\mathcal{H}_{\varphi}$ and its commutators in the Dunkl setting. We introduce mixed-norm radial-angular Dunkl central Morrey spaces,
together with their $\lambda$-central counterparts. For the operator $\mathcal{H}_{\varphi}$, we derive the necessary and sufficient conditions on the non-negative weight $\varphi$ that ensure its boundedness on these spaces, and explicitly determine the exact operator norms. For the commutator $\mathcal{H}_{\varphi,b}$, we establish the necessary and sufficient conditions on the weight $\varphi$ such that the commutator is bounded for all symbols $b$ in the mixed-norm radial-angular Dunkl central bounded mean oscillation space $\mathrm{CMO}_{\vec p, \vec q, k}(\mathbb{R}^n)$. Furthermore, for all symbols $b$ in the mixed-norm radial-angular Dunkl $\lambda$-central bounded mean oscillation space $\mathrm{CMO}_{\vec p, \vec q,\lambda, k}(\mathbb{R}^n)$ with positive index $\lambda>0$, we also obtain a valid sufficient condition for boundedness.
\vspace{1mm}\baselineskip 12pt

\no{\small\bf Keywords}\ \ Weighted Hardy--Littlewood operator; Commutators; Mixed-norm radial-angular Dunkl $\lambda$-central bounded mean oscillation space; Mixed-norm radial-angular Dunkl $\lambda$-central Morrey spaces.
\par

\vspace{2mm}

\no{\small\bf MR(2020) Subject Classification\ \ {\rm 44A15; 42B35; 26D15}}

}}

\baselineskip 15pt

\section{Introduction}\label{sec1}

Motivated by the study of spherical functions on Riemannian symmetric spaces, Dunkl \cite{Dunkl4,Dunkl5,Dunkl6,Dunkl7,Dunkl8} introduced a family of commuting differential-difference operators associated with finite reflection groups. These operators, together with their corresponding weighted measures, laid the foundation of Dunkl theory and naturally generalized classical Euclidean harmonic analysis. Real analytic methods involving root systems and Dunkl distances play a central role in this field. For systematic overviews and extensive applications, see \cite{Rosler9,Anker10,Gallardo11,Guliyev12,Ivanov13,Liu14,Mammadov15,Muslimova16}.

Let $\varphi$ be a non-negative measurable function on $(0,1]$. The weighted Hardy--Littlewood operator is defined by
\begin{equation}
\mathcal{H}_{\varphi}f(x)
=\int_{0}^{1} f(tx)\varphi(t)\,\mathrm{d}t,
\end{equation}
whenever the integral is well defined. For a locally integrable function $b$, the commutator generated by $\mathcal{H}_{\varphi}$ and $b$ is given by
\begin{equation}
\mathcal{H}_{\varphi,b}f(x)
=b(x)\mathcal{H}_{\varphi}f(x)-\mathcal{H}_{\varphi}(bf)(x).
\end{equation}
In the one-dimensional case with $\varphi\equiv1$, $\mathcal{H}_{\varphi}$ reduces to the classical Hardy operator, and its generalizations, such as the Hausdorff operators and their corresponding commutators, have attracted considerable research interest over recent decades; see \cite{Xiao2001,andersen2003,moricz2005,chen2012,chen2013,wu2015,karapetyants2020,An2023,Wei2025}.

Morrey \cite{Morrey1938} introduced Morrey spaces to study the regularity of elliptic PDEs. This framework has since been generalized in two distinct directions. First, building on the mixed-norm Lebesgue spaces of Benedek and Panzone \cite{BenedekPanzone1961}, Nogayama \cite{Nogayama2019Mixed} developed mixed Morrey spaces. Second, to characterize local behavior near the origin, Lu and Yang \cite{LuYang1995} and Alvarez et al.\ \cite{Alvarez2000} introduced central BMO and central Morrey spaces. These two concepts were recently unified into mixed central Morrey spaces \cite{Wei2022JMI,LuZhou2024Campanato}. In the Dunkl setting, the intrinsic product form of the measure $\mathrm{d}\mu_k$ makes the study of such mixed-norm spaces particularly natural.

In this paper, we establish sharp boundedness criteria for the weighted Hardy--Littlewood operator $\mathcal{H}_{\varphi}$ and its commutator $\mathcal{H}_{\varphi,b}$ in the Dunkl setting. We obtain necessary and sufficient conditions on the weight $\varphi$ for $\mathcal{H}_{\varphi}$ to be bounded on mixed-norm radial-angular Dunkl central Morrey spaces and mixed-norm radial-angular Dunkl $\lambda$-central Morrey spaces. For the commutator $\mathcal{H}_{\varphi,b}$, we establish analogous necessary and sufficient conditions when the symbol $b$ belongs to the mixed-norm radial-angular Dunkl central BMO space. Furthermore, we provide a valid sufficient condition for its boundedness when $b$ is in the corresponding $\lambda$-central BMO space.

The paper is organized as follows. Section~\ref{sec2} recalls basic Dunkl theory and introduces the relevant function spaces. Section~\ref{sec3} establishes sharp bounds for $\mathcal{H}_{\varphi}$, and Section~\ref{sec4} characterizes the boundedness of the commutator $\mathcal{H}_{\varphi,b}$.

Throughout the paper, $C$ denotes a generic positive constant that may change from line to line. For a measurable set $E \subset \mathbb{R}^n$, $\chi_E$ denotes its characteristic function. We write $B(a,R)$ for the open ball centered at $a \in \mathbb{R}^n$ with radius $R>0$, and $\mathcal{M}(\mathbb{R}^n)$ for the class of all Lebesgue measurable functions on $\mathbb{R}^n$. For a vector exponent $\vec{p}=(p_1,\ldots,p_m)$, the condition $0<\vec{p}<\infty$ means $0<p_i<\infty$ for every $i=1,\ldots,m$. The symbol $A\lesssim B$ denotes that there exists a positive constant $C$ such that $A\leq CB$.

\section{Preliminaries}\label{sec2}

\subsection{Basics of Dunkl Theory}

Dunkl theory relies on root systems and finite reflection groups acting on $\mathbb{R}^n$. For a non-zero vector $\eta \in \mathbb{R}^n$, the reflection $\sigma_\eta$ across the hyperplane $H = \{ x \in \mathbb{R}^n : \langle x, \eta \rangle = 0 \}$ is defined by
$$
\sigma_\eta(x) = x - 2\frac{\langle x, \eta \rangle}{\langle \eta, \eta \rangle} \eta.
$$

A finite set $R \subset \mathbb{R}^n \setminus \{0\}$ is called a root system if it satisfies (i) $\sigma_\eta R = R$ for all $\eta \in R$, and (ii) $R \cap \eta\mathbb{R} = \{\pm \eta\}$. The group generated by the reflections $\{\sigma_\eta\}_{\eta \in R}$ is the reflection group (or Coxeter group) associated with $R$.

Let $e_1,\dots,e_n$ be the standard orthonormal basis of $\mathbb{R}^n$. The reflection with respect to the hyperplane orthogonal to $e_j$ is given by
$$
\sigma_j(x) = x - 2\langle x, e_j \rangle e_j = (x_1, \dots, -x_j, \dots, x_n).
$$
These reflections generate the Weyl group $G$ associated with the root system $R_e = \{ \pm e_1, \dots, \pm e_n \}$, explicitly defined as
$$
G=\left\{\prod\limits_{j=1}^n\sigma_j^{a_j}:(a_1,\ldots, a_n)\in\{0,1\}^n, \ \text{where} \ \sigma_j^{0}=id\right\}.
$$
Note that $G$ is isomorphic to the direct product of cyclic groups, i.e., $G \simeq \mathbb{Z}_2^n$ (see \cite{Cheng1959}).

A multiplicity function $k = (k_1, \dots, k_n)$ is a non-negative, $G$-invariant function on the root system. It serves as a continuous parameter that generalizes the integer dimensions of root spaces in classical Riemannian symmetric spaces. The associated Dunkl measure is defined by
$$
\mathrm{d}\mu_k(x) = \mu_k(x)\,\mathrm{d}x, \quad\text{where}\ \mu_k(x) = \prod_{j=1}^n |\langle e_j, x \rangle|^{2k_j} = \prod_{j=1}^n |x_j|^{2k_j}.
$$
This measure is homogeneous of degree $N_k = n + 2\sum_{j=1}^{n} k_j$, meaning
$$
\mathrm{d}\mu_k(\lambda x) = |\lambda|^{N_k} \mathrm{d}\mu_k(x), \quad \forall \lambda \in \mathbb{R}\setminus\{0\}.
$$
For an open ball $B(x, R)$ of radius $R$ centered at $x$, we have
$$
\mu_k(B(\lambda x, |\lambda| R)) = |\lambda|^{N_k} \mu_k(B(x, R)), \quad \mu_k(B(x, R)) \approx R^n \prod_{j=1}^n (|x_j| + R)^{2k_j}.
$$
Consequently, $\mathrm{d}\mu_k$ is a doubling measure: there exists a constant $C > 0$ such that $\mu_k(B(x, 2R)) \leq C \mu_k(B(x, R))$. Furthermore, there exists $C \geq 1$ such that for any $x \in \mathbb{R}^n$ and $0 < R_1 \leq R_2$,
$$
C^{-1}\left( \frac{R_2}{R_1} \right)^{n} \leq \frac{\mu_k(B(x, R_2))}{\mu_k(B(x, R_1))} \leq C\left( \frac{R_2}{R_1} \right)^{N_k}.
$$

\begin{remark}
The surface area $\omega_{n,k}$ and the volume $\nu_{n,k}$ of the unit ball under the Dunkl measure are given by
\begin{align*}
\omega_{n,k}&= \int_{\mathbb{S}^{n-1}} \,\mathrm{d}\sigma_k(\xi) = \frac{2 \prod_{j=1}^n \Gamma\left(k_j + \frac12\right)}{\Gamma\left(\frac{N_k}{2}\right)},\\
\nu_{n,k} &= \omega_{n,k} \int_0^1 \rho^{N_k-1} \,\mathrm{d}\rho = \frac{\omega_{n,k}}{N_k}.
\end{align*}
If $k_j = 0$ for all $j$, then $\omega_{n,0} = 2\pi^{n/2}/\Gamma(n/2)$ and $\nu_{n,0} = \pi^{n/2}/\Gamma(n/2 + 1)$, which recover the standard geometric constants of the Euclidean space.
\end{remark}

The mutually commuting Dunkl differential-difference operators \cite{Dunkl5} are defined by
$$
\mathcal{D}_j f(x) = \partial_j f(x) + k_j \frac{f(x) - f(\sigma_j(x))}{x_j}, \quad j = 1, \dots, n.
$$

When $k_j = 0$ for all $j$, Dunkl theory completely reduces to the classical Euclidean setting: the Dunkl measure $\mathrm{d}\mu_k(x)$ becomes the Lebesgue measure $\mathrm{d}x$, and the difference term vanishes, reducing $\mathcal{D}_j$ to the standard partial derivative $\partial_j$. Thus, Dunkl theory establishes a framework that extends classical harmonic analysis from spaces with rotational symmetry to those with reflection symmetry. This theory has found extensive applications, notably in the analytic continuation of spherical functions on Riemannian symmetric spaces and the probabilistic construction of stochastic evolution models (see \cite{Rosler9,Heckman1991,RV1998}).

\subsection{Function Spaces in the Dunkl Setting}

We first recall several function spaces with mixed norms which will be used throughout the paper. Mixed Lebesgue spaces were introduced by Benedek and Panzone \cite{BenedekPanzone1961} as a natural extension of the classical Lebesgue spaces. Let
\[
\vec p=(p_1,\ldots,p_n), \qquad 0<\vec p\leq \infty.
\]
For $0<\vec p<\infty$, the mixed Lebesgue norm of a measurable function $f$ on $\mathbb{R}^n$ is defined by
\begin{equation}
\|f\|_{L^{\vec p}(\mathbb{R}^n)}
=
\left(
\int_{\mathbb{R}}
\cdots
\left(
\int_{\mathbb{R}}
\left(
\int_{\mathbb{R}}
|f(x_1,\ldots,x_n)|^{p_1}\,\mathrm{d}x_1
\right)^{\frac{p_2}{p_1}}
\mathrm{d}x_2
\right)^{\frac{p_3}{p_2}}
\cdots
\mathrm{d}x_n
\right)^{\frac{1}{p_n}} .
\end{equation}
If $p_j=\infty$ for some $j$, the corresponding integral is replaced by the essential supremum in the usual way. The mixed Lebesgue space $L^{\vec p}(\mathbb{R}^n)$ consists of all measurable functions $f$ such that $\|f\|_{L^{\vec p}(\mathbb{R}^n)}<\infty$.

Morrey spaces were introduced by Morrey \cite{Morrey1938} in the study of regularity problems for elliptic partial differential equations. Replacing the usual Lebesgue norm in the classical Morrey space by a mixed Lebesgue norm leads to the mixed Morrey spaces introduced by Nogayama \cite{Nogayama2019Mixed}. The corresponding central version was considered by Wei \cite{Wei2022JMI}. More precisely, let $\vec q=(q_1,\ldots,q_n)\in(0,\infty]^n$ and $p\in(0,\infty]$ satisfy
\[
\sum_{j=1}^{n}\frac{1}{q_j}\geq \frac{n}{p}.
\]
The mixed central Morrey space $\dot{M}_{\vec q}^{p}(\mathbb{R}^n)$ is defined as the set of all $f\in\mathcal{M}(\mathbb{R}^n)$ such that
\begin{equation}
\|f\|_{\dot{M}_{\vec q}^{p}(\mathbb{R}^n)}
=\sup_{R>0}|B(0,R)|^{\frac{1}{p}-\frac{1}{n}\sum\limits_{j=1}^{n}\frac{1}{q_j}}\|f\chi_{B(0,R)}\|_{L^{\vec q}(\mathbb{R}^n)}<\infty .
\end{equation}

Following the classical $\lambda$-central Morrey spaces and $\lambda$-central bounded mean oscillation spaces of Alvarez, Guzm\'an-Partida and Lakey \cite{Alvarez2000}, and their mixed-norm extensions considered in \cite{LuZhou2023Fractional,LuZhou2024Campanato},  the mixed $\lambda$-central Morrey space is defined as follows.

\begin{definition}
Let $1<\vec p<\infty$ and $\lambda\in\mathbb{R}$. The mixed $\lambda$-central Morrey space $\dot{B}^{\vec p,\lambda}(\mathbb{R}^n)$ consists of all $f\in\mathcal{M}(\mathbb{R}^n)$ such that
\begin{equation}
\|f\|_{\dot{B}^{\vec p,\lambda}(\mathbb{R}^n)}
=\sup_{R>0}\frac{\|f\chi_{B(0,R)}\|_{L^{\vec p}(\mathbb{R}^n)}}{|B(0,R)|^{\lambda}\|\chi_{B(0,R)}\|_{L^{\vec p}(\mathbb{R}^n)}}<\infty .
\end{equation}
\end{definition}

We next recall the corresponding mean oscillation spaces. For a ball $B\subset\mathbb{R}^n$, we write
\[
f_B=\frac{1}{|B|}\int_B f(y)\,\mathrm{d}y .
\]

\begin{definition}
Let $1<\vec p<\infty$. The mixed bounded mean oscillation space $\mathrm{BMO}_{\vec p}(\mathbb{R}^n)$ is the space of all locally integrable functions $f$ such that
\begin{equation}
\|f\|_{\mathrm{BMO}_{\vec p}(\mathbb{R}^n)}
=
\sup_{B}
\frac{
\|(f-f_B)\chi_B\|_{L^{\vec p}(\mathbb{R}^n)}
}{
\|\chi_B\|_{L^{\vec p}(\mathbb{R}^n)}
}
<\infty ,
\end{equation}
where the supremum is taken over all balls $B\subset\mathbb{R}^n$.
\end{definition}

The central bounded mean oscillation space was introduced by Lu and Yang \cite{LuYang1995}. Its mixed-norm version was used in the study of Hardy-type operators and mixed central Morrey spaces; see, for instance, \cite{Wei2021MixedHerz,Wei2022JMI}.

\begin{definition}
Let $1<\vec p<\infty$ and $\lambda<1/n$. The mixed $\lambda$-central bounded mean oscillation space $\mathrm{CMO}_{\vec p,\lambda}(\mathbb{R}^n)$ consists of all locally integrable functions $f$ such that
\begin{equation}
\|f\|_{\mathrm{CMO}_{\vec p,\lambda}(\mathbb{R}^n)}
=
\sup_{R>0}
\frac{
\|(f-f_{B(0,R)})\chi_{B(0,R)}\|_{L^{\vec p}(\mathbb{R}^n)}
}{
|B(0,R)|^{\lambda}
\|\chi_{B(0,R)}\|_{L^{\vec p}(\mathbb{R}^n)}
}
<\infty .
\end{equation}
When $\lambda=0$, we simply write $\mathrm{CMO}_{\vec p}(\mathbb{R}^n)$.
\end{definition}

The mixed radial-angular spaces $L^{p}_{\rm rad}L^{\tilde{p}}_{\rm ang}(\mathbb R^n)$, as a extension of classical Lebesgue spaces, were defined as follows:

Let $\mathbb{S}^{n-1}$ be the unit sphere in $\mathbb R^n$, $n\geq 2$, with Lebesgue measure $\mathrm{d}\sigma=\mathrm{d}\sigma(\cdot)$,
$$
\|f\|_{L^{p}_{\rm rad}L^{\tilde{p}}_{\rm ang}(\mathbb R^n)}:=\left(\int_{0}^{\infty}\|f(\rho\cdot)\|^p_{L^{\tilde{p}}(\mathbb{S}^{n-1})}\rho^{n-1}\,\mathrm{d}\rho\right)^{\frac{1}{p}},\quad  1\leq p, \tilde{p}\leq\infty,
$$
and when $p=\infty$ or $\tilde{p}=\infty$, we just need to make the usual modifications in the above definition, but we do not use these cases in the current work.

The boundedness for some classical operators in harmonic analysis on mixed radial-angular spaces were investigated successively in \cite{CL,DJ,DL1,LF,LLW,LLW2,LW,LRW,LiuRonghui2026,LiuRonghui2025}.

Inspired by the above ideas, we extend the classical function spaces and operators to the Dunkl setting.

Let the space $\mathbb{R}^n$ be decomposed as the direct product $\mathbb{R}^{n_1} \times \cdots \times \mathbb{R}^{n_m}$, where $\sum\limits_{i=1}^m n_i = n$ and $n\geq 2$. For any $x = (x_1, \dots, x_m) \in \mathbb{R}^n$ with $x_i \in \mathbb{R}^{n_i}$, we write $x_i = \rho_i \omega_i$ in polar coordinates, where $\rho_i > 0$ and $\omega_i \in \mathbb{S}^{n_i-1}$. Let the multiplicity function $k = (k_1, \dots, k_n)$ be correspondingly decomposed into $m$ blocks as $k = (k^{(1)}, \dots, k^{(m)})$, where each sub-vector $k^{(i)}= (k_1^{(i)}, \dots, k_{n_i}^{(i)})$ acts on the block $\mathbb{R}^{n_i}$. The homogeneous dimension of the $i$-th block is defined by
$$
N_{k^{(i)}} = n_i + 2 \sum_{v=1}^{n_i} k^{(i)}_v,\quad i=1\cdots m,
$$
where $k^{(i)}_v$ denotes the $v$-th component of $k^{(i)}$. Consequently, the total homogeneous dimension of $\mathbb{R}^n$ under the Dunkl measure naturally satisfies $N_k = \sum\limits_{i=1}^m N_{k^{(i)}}$.

\begin{definition}
For $1< \vec p = (p_1, \ldots, p_m) < \infty$ and $1< \vec q = (q_1, \ldots, q_m) < \infty$, the mixed-norm radial-angular Dunkl--Lebesgue space $L_{\mathrm{rad}, k}^{\vec p} L_{\mathrm{ang}, k}^{\vec q}(\mathbb{R}^n)$ consists of all measurable functions $f:\mathbb{R}^n\to\mathbb{C}$ with finite norm:
\begin{align*}
&\|f\|_{L_{\mathrm{rad}, k}^{\vec p} L_{\mathrm{ang}, k}^{\vec q}(\mathbb{R}^n)}\\
&= \Biggl( \int_0^\infty \Biggl( \int_{\mathbb{S}^{n_m-1}} \cdots \Biggl( \int_0^\infty \Biggl( \int_{\mathbb{S}^{n_2-1}} \Biggl( \int_0^\infty \Biggl( \int_{\mathbb{S}^{n_1-1}} |f(\rho_1\omega_1, \dots, \rho_m\omega_m)|^{q_1} \,\mathrm{d}\sigma_{k^{(1)}}(\omega_1)
\Biggr)^{\frac{p_1}{q_1}}\\
&\quad \times\rho_1^{N_{k^{(1)}}-1} \,\mathrm{d}\rho_1 \Biggr)^{\frac{q_2}{p_1}} \,\mathrm{d}\sigma_{k^{(2)}}(\omega_2) \Biggr)^{\frac{p_2}{q_2}}
\rho_2^{N_{k^{(2)}}-1} \,\mathrm{d}\rho_2 \Biggr)^{\frac{q_3}{p_2}} \cdots \,\mathrm{d}\sigma_{k^{(m)}}(\omega_m) \Biggr)^{\frac{p_m}{q_m}}
\rho_m^{N_{k^{(m)}}-1} \,\mathrm{d}\rho_m \Biggr)^{\frac{1}{p_m}}< \infty.
\end{align*}
\end{definition}

\begin{definition}
Let $\vec{p} \in (1, \infty)^m$, $\vec{q} \in (1, \infty)^m$, and $0 < r <  \infty$ satisfy $\sum\limits_{i=1}^m \frac{N_{k^{(i)}}}{p_i} \geq \frac{N_k}{r}$. The mixed-norm radial-angular Dunkl central Morrey space $\dot{M}_{\vec{p},\vec{q},k}^r(\mathbb{R}^n)$ consists of all measurable functions $f$ such that
\[
\|f\|_{\dot{M}_{\vec{p},\vec{q},k}^r(\mathbb{R}^n)} = \sup_{R>0} |B(0,R)|_k^{\frac{1}{r} - \frac{1}{N_k}\sum\limits_{i=1}^m \frac{N_{k^{(i)}}}{p_i}} \big\|f\chi_{B(0,R)}\big\|_{L_{\mathrm{rad},k}^{\vec{p}} L_{\mathrm{ang},k}^{\vec{q}}(\mathbb{R}^n)} < \infty.
\]
\end{definition}

\begin{remark}
By letting $x = \rho \xi$ with $\rho \in (0, R)$ and $\xi \in \mathbb{S}^{n-1}$, we have
$$
|B(0, R)|_k = \int_{B(0, R)} \mathrm{d}\mu_k(x) = \left( \int_0^R \rho^{N_k-1} \,\mathrm{d}\rho \right) \left( \int_{\mathbb{S}^{n-1}} \mathrm{d}\sigma_k(\xi) \right) = \frac{R^{N_k}}{N_k} \omega_{n,k}.
$$
Consequently, the volume is given by $|B(0, R)|_k = \nu_{n,k} R^{N_k}$ and $|B(0,tR)|_k=t^{N_k}|B(0,R)|_k$.
\end{remark}

\begin{remark}
This remark establishes the scaling identity of $\chi_{B(0,R)}$ to rigorously justify the normalization exponent of $\dot{M}_{\vec{p},\vec{q},k}^r(\mathbb{R}^n)$.
\begin{align*}
&\|\chi_{B(0,R)}\|_{L^{\vec p}_{\mathrm{rad},k} L^{\vec q}_{\mathrm{ang},k}(\mathbb{R}^n)} \\
&= \Biggl( \int_0^\infty \cdots \Biggl( \int_{\mathbb{S}^{n_2-1}} \Biggl( \int_0^\infty \chi_{\{\sum\limits_{i=1}^m \rho_i^2 < R^2\}}
\Biggl( \int_{\mathbb{S}^{n_1-1}} 1 \,\mathrm{d}\sigma_{k^{(1)}}(\omega_1) \Biggr)^{\frac{p_1}{q_1}}
\rho_1^{N_{k^{(1)}}-1} \,\mathrm{d}\rho_1 \Biggr)^{\frac{q_2}{p_1}} \,\mathrm{d}\sigma_{k^{(2)}}(\omega_2) \Biggr)^{\frac{p_2}{q_2}} \cdots \Biggr)^{\frac{1}{p_m}} \\
&= \Biggl( \int_0^\infty \cdots \Biggl( \int_{\mathbb{S}^{n_2-1}} \omega_{n_1, k^{(1)}}^{\frac{q_2}{q_1}} \Biggl( \int_0^\infty \chi_{\{\sum\limits_{i=1}^m \rho_i^2 < R^2\}}
\rho_1^{N_{k^{(1)}}-1} \,\mathrm{d}\rho_1 \Biggr)^{\frac{q_2}{p_1}} \,\mathrm{d}\sigma_{k^{(2)}}(\omega_2) \Biggr)^{\frac{p_2}{q_2}} \cdots \Biggr)^{\frac{1}{p_m}} \\
&= \Biggl( \int_0^\infty \cdots \Biggl( \omega_{n_1, k^{(1)}}^{\frac{q_2}{q_1}} \omega_{n_2, k^{(2)}} \Biggl( \int_0^\infty \chi_{\{\sum\limits_{i=1}^m \rho_i^2 < R^2\}}
\rho_1^{N_{k^{(1)}}-1} \,\mathrm{d}\rho_1 \Biggr)^{\frac{q_2}{p_1}} \Biggr)^{\frac{p_2}{q_2}} \rho_2^{N_{k^{(2)}}-1} \,\mathrm{d}\rho_2 \cdots \Biggr)^{\frac{1}{p_m}} \\
&= \Biggl( \int_0^\infty \cdots \omega_{n_1, k^{(1)}}^{\frac{p_2}{q_1}} \omega_{n_2, k^{(2)}}^{\frac{p_2}{q_2}} \Biggl( \int_0^\infty \chi_{\{\sum\limits_{i=1}^m \rho_i^2 < R^2\}}
\rho_1^{N_{k^{(1)}}-1} \,\mathrm{d}\rho_1 \Biggr)^{\frac{p_2}{p_1}} \rho_2^{N_{k^{(2)}}-1} \,\mathrm{d}\rho_2 \cdots \Biggr)^{\frac{1}{p_m}} \\
&= \left( \prod_{i=1}^m \big(\omega_{n_i, k^{(i)}}\big)^{\frac{1}{q_i}} \right) \Biggl( \int_0^\infty \cdots \Biggl( \int_0^\infty \chi_{\{\sum\limits_{i=1}^m \rho_i^2 < R^2\}}
\rho_1^{N_{k^{(1)}}-1} \,\mathrm{d}\rho_1 \Biggr)^{\frac{p_2}{p_1}} \cdots \,\mathrm{d}\rho_m \Biggr)^{\frac{1}{p_m}}.
\end{align*}
Next, we perform the radial change of variables $\rho_i = R\tau_i$. We further obtain
\begin{align*}
&\|\chi_{B(0,R)}\|_{L^{\vec p}_{\mathrm{rad},k} L^{\vec q}_{\mathrm{ang},k}(\mathbb{R}^n)}\\
&= \prod\limits_{i=1}^m \big(\omega_{n_i, k^{(i)}}\big)^{\frac{1}{q_i}} \Biggl( \int_0^\infty \cdots \Biggl( \int_0^\infty \chi_{\{\sum\limits_{i=1}^m \tau_i^2 < 1\}}
(R\tau_1)^{N_{k^{(1)}}-1} R\,\mathrm{d}\tau_1 \Biggr)^{\frac{p_2}{p_1}} \cdots R\,\mathrm{d}\tau_m \Biggr)^{\frac{1}{p_m}} \\
&=\prod\limits_{i=1}^m \big(\omega_{n_i, k^{(i)}}\big)^{\frac{1}{q_i}} \Biggl( \int_0^\infty \cdots R^{N_{k^{(1)}}\frac{p_2}{p_1} + N_{k^{(2)}}} \Biggl( \int_0^\infty \chi_{\{\sum\limits_{i=1}^m \tau_i^2 < 1\}}
\tau_1^{N_{k^{(1)}}-1} \,\mathrm{d}\tau_1 \Biggr)^{\frac{p_2}{p_1}} \tau_2^{N_{k^{(2)}}-1} \,\mathrm{d}\tau_2 \cdots \Biggr)^{\frac{1}{p_m}} \\
&= \prod\limits_{i=1}^m \big(\omega_{n_i, k^{(i)}}\big)^{\frac{1}{q_i}} \prod_{i=1}^m R^{\frac{N_{k^{(i)}}}{p_i}} \Biggl( \int_0^\infty \cdots \Biggl( \int_0^\infty \chi_{\{\sum\limits_{i=1}^m \tau_i^2 < 1\}}
\tau_1^{N_{k^{(1)}}-1} \,\mathrm{d}\tau_1 \Biggr)^{\frac{p_2}{p_1}} \cdots \,\mathrm{d}\tau_m \Biggr)^{\frac{1}{p_m}} \\
&= R^{\sum\limits_{i=1}^m \frac{N_{k^{(i)}}}{p_i}}\|\chi_{B(0,1)}\|_{L^{\vec p}_{\mathrm{rad},k} L^{\vec q}_{\mathrm{ang},k}(\mathbb{R}^n)}.
\end{align*}

Recall that the Dunkl volume of the ball is given by $|B(0,R)|_k = \nu_{n,k} R^{N_k}$, which yields $R = (\nu_{n,k})^{-\frac{1}{N_k}} |B(0,R)|_k^{\frac{1}{N_k}}$. So we have
\begin{align*}
\|\chi_{B(0,R)}\|_{L^{\vec p}_{\mathrm{rad},k} L^{\vec q}_{\mathrm{ang},k}(\mathbb{R}^n)}
&= \|\chi_{B(0,1)}\|_{L^{\vec p}_{\mathrm{rad},k} L^{\vec q}_{\mathrm{ang},k}(\mathbb{R}^n)} \left[ (\nu_{n,k})^{-\frac{1}{N_k}} |B(0,R)|_k^{\frac{1}{N_k}} \right]^{\sum\limits_{i=1}^m \frac{N_{k^{(i)}}}{p_i}} \\
&= (\nu_{n,k})^{-\frac{1}{N_k} \sum\limits_{i=1}^m \frac{N_{k^{(i)}}}{p_i}} |B(0,R)|_k^{\frac{1}{N_k}\sum\limits_{i=1}^m \frac{N_{k^{(i)}}}{p_i}}\|\chi_{B(0,1)}\|_{L^{\vec p}_{\mathrm{rad},k} L^{\vec q}_{\mathrm{ang},k}(\mathbb{R}^n)}.
\end{align*}
\end{remark}

\begin{remark}
Let $\vec q=(q_1,\ldots,q_n)\in(1,\infty)^n$ and $r\in(0,\infty)$ satisfying $\sum\limits_{j=1}^{n}\frac{1+2k_j}{q_j} \geq \frac{N_k}{r}$.
By restricting the supremum to balls centered at the origin, we obtain the mixed Dunkl central Morrey space $\dot{M}_{\vec q, k}^{r}(\mathbb{R}^n)$ with the norm
$$
\left\| f \right\|_{\dot{M}_{\vec q, k}^{r}(\mathbb{R}^n)}
=\sup_{R>0}|B(0, R)|_k^{\frac{1}{r}-\frac{1}{N_k}\left(\sum\limits_{j=1}^{n}\frac{1+2k_j}{q_j}\right)}\left\| f\chi_{B(0, R)} \right\|_{L^{\vec q}_k(\mathbb{R}^n)}.
$$
\end{remark}

Note that if we do not distinguish between the radial and angular exponents (i.e., $\vec{p} = \vec{q}$), and we assume $m=n$ (each block is 1-dimensional), the space $\dot{M}_{\vec{p},\vec{q},k}^r(\mathbb{R}^n)$ reduces to the standard mixed Dunkl central Morrey space $\dot{M}_{\vec{q},k}^r(\mathbb{R}^n)$.

\begin{definition}
Let $1<\vec p, \vec q<\infty$ and $\lambda\in\mathbb{R}$. The mixed-norm radial-angular Dunkl $\lambda$-central Morrey space $\dot{B}^{\vec p, \vec q,\lambda}_k(\mathbb{R}^n)$ is defined as the set of all measurable functions $f$ such that
$$
\left\| f \right\|_{\dot{B}^{\vec p, \vec q,\lambda}_k(\mathbb{R}^n)}
=\sup_{R>0}\frac{\left\| f\chi_{B(0, R)} \right\|_{L_{\mathrm{rad}, k}^{\vec p} L_{\mathrm{ang}, k}^{\vec q}(\mathbb{R}^n)}}
{|B(0, R)|_k^{\lambda}\left\| \chi_{B(0, R)} \right\|_{L_{\mathrm{rad}, k}^{\vec p} L_{\mathrm{ang}, k}^{\vec q}(\mathbb{R}^n)}} < \infty.
$$
\end{definition}

\begin{remark}
If $\lambda=0$, for $1<\vec p, \vec q<\infty$, the mixed-norm radial-angular Dunkl bounded mean oscillation space $\mathrm{BMO}_{\vec p, \vec q, k}(\mathbb{R}^n)$ is defined as the set of all measurable functions $f$ such that
$$
\left\| f \right\|_{\mathrm{BMO}_{\vec p, \vec q, k}(\mathbb{R}^n)}
=\sup_{B}\frac{\left\| (f-f_{B, k})\chi_B \right\|_{L_{\mathrm{rad}, k}^{\vec p} L_{\mathrm{ang}, k}^{\vec q}(\mathbb{R}^n)}}
{\left\| \chi_B \right\|_{L_{\mathrm{rad}, k}^{\vec p} L_{\mathrm{ang}, k}^{\vec q}(\mathbb{R}^n)}} < \infty,
$$
where the supremum is taken over all balls $B \subset \mathbb{R}^n$ and we write
\[
f_{B, k}=\frac{1}{|B|_k}\int_B f(y)\,\mathrm{d}\mu_k(y).
\]
\end{remark}

\begin{definition}
Let $\lambda<1/N_k$ and $1<\vec p, \vec q<\infty$. The mixed-norm radial-angular Dunkl $\lambda$-central bounded mean oscillation space $\mathrm{CMO}_{\vec p, \vec q,\lambda, k}(\mathbb{R}^n)$ is defined by
$$
\left\| f \right\|_{\mathrm{CMO}_{\vec p, \vec q,\lambda, k}(\mathbb{R}^n)}
=\sup_{R>0}\frac{\left\| (f-f_{B(0, R), k})\chi_{B(0, R)} \right\|_{L_{\mathrm{rad}, k}^{\vec p} L_{\mathrm{ang}, k}^{\vec q}(\mathbb{R}^n)}}
{|B(0, R)|_k^{\lambda}\left\| \chi_{B(0, R)} \right\|_{L_{\mathrm{rad}, k}^{\vec p} L_{\mathrm{ang}, k}^{\vec q}(\mathbb{R}^n)}} < \infty.
$$
\end{definition}

\begin{remark}
If $\lambda=0$, we get the mixed-norm radial-angular Dunkl central bounded mean oscillation space $\mathrm{CMO}_{\vec p, \vec q, k}(\mathbb{R}^n)$ with the following norm
$$
\left\| f \right\|_{\mathrm{CMO}_{\vec p, \vec q, k}(\mathbb{R}^n)}
=\sup_{R>0}\frac{\left\| (f-f_{B(0, R), k})\chi_{B(0, R)} \right\|_{L_{\mathrm{rad}, k}^{\vec p} L_{\mathrm{ang}, k}^{\vec q}(\mathbb{R}^n)}}
{\left\| \chi_{B(0, R)} \right\|_{L_{\mathrm{rad}, k}^{\vec p} L_{\mathrm{ang}, k}^{\vec q}(\mathbb{R}^n)}} < \infty.
$$
\end{remark}
\subsection{The mixed-norm radial-angular Dunkl central Morrey space}

In this section, we discuss the fundamental properties and typical examples of the mixed-norm radial-angular Dunkl central Morrey spaces $\dot{M}_{\vec{p},\vec{q},k}^r(\mathbb{R}^n)$.

\begin{remark}\label{rem:banach_space}
The mixed Morrey space $\dot{M}_{\vec{p},\vec{q},k}^r(\mathbb{R}^n)$ is a Banach space for $1 < \vec{p}, \vec{q} < \infty$ and $0 < r < \infty$. Although the proof follows standard functional analysis techniques, we briefly outline it for completeness. First, the triangle inequality is verified by Minkowski's inequality. For $f, g \in \dot{M}_{\vec{p},\vec{q},k}^r(\mathbb{R}^n)$,
\begin{align*}
\|f + g\|_{\dot{M}_{\vec{p},\vec{q},k}^r(\mathbb{R}^n)}
&= \sup_{R>0} |B(0,R)|_k^{\frac{1}{r} - \frac{1}{N_k}\sum\limits_{i=1}^m \frac{N_{k^{(i)}}}{p_i}} \big\|(f + g)\chi_{B(0,R)}\big\|_{L_{\mathrm{rad},k}^{\vec{p}} L_{\mathrm{ang},k}^{\vec{q}}} \\
&\le \sup_{R>0} |B(0,R)|_k^{\frac{1}{r} - \frac{1}{N_k}\sum\limits_{i=1}^m \frac{N_{k^{(i)}}}{p_i}} \left( \big\|f\chi_{B(0,R)}\big\|_{L_{\mathrm{rad},k}^{\vec{p}} L_{\mathrm{ang},k}^{\vec{q}}} + \big\|g\chi_{B(0,R)}\big\|_{L_{\mathrm{rad},k}^{\vec{p}} L_{\mathrm{ang},k}^{\vec{q}}} \right) \\
&\le \|f\|_{\dot{M}_{\vec{p},\vec{q},k}^r(\mathbb{R}^n)} + \|g\|_{\dot{M}_{\vec{p},\vec{q},k}^r(\mathbb{R}^n)}.
\end{align*}
Positivity and absolute homogeneity are evident. In particular, $\|f\|_{\dot{M}_{\vec{p},\vec{q},k}^r(\mathbb{R}^n)} = 0 \Rightarrow f = 0$ a.e.\ on every $B(0, R)$, hence $f = 0$ a.e.\ on $\mathbb{R}^n$. Thus, $\dot{M}_{\vec{p},\vec{q},k}^r(\mathbb{R}^n)$ is a normed space.

To check completeness, let $\{f_j\}_{j=1}^\infty \subset \dot{M}_{\vec{p},\vec{q},k}^r(\mathbb{R}^n)$ be a sequence such that $\sum\limits_{j=1}^\infty \|f_j\|_{\dot{M}_{\vec{p},\vec{q},k}^r(\mathbb{R}^n)} < \infty$. For any fixed ball $B(0,R)$ with $R>0$, it follows from the definition of the Morrey norm that
\begin{align*}
\sum_{j=1}^\infty \|f_j\chi_{B(0,R)}\|_{L_{\mathrm{rad},k}^{\vec{p}} L_{\mathrm{ang},k}^{\vec{q}}}
&\le \sum\limits_{j=1}^\infty |B(0,R)|_k^{-\frac{1}{r} + \frac{1}{N_k}\sum\limits_{i=1}^m \frac{N_{k^{(i)}}}{p_i}} \|f_j\|_{\dot{M}_{\vec{p},\vec{q},k}^r(\mathbb{R}^n)} \\
&= |B(0,R)|_k^{-\frac{1}{r} + \frac{1}{N_k}\sum\limits_{i=1}^m \frac{N_{k^{(i)}}}{p_i}} \sum\limits_{j=1}^\infty \|f_j\|_{\dot{M}_{\vec{p},\vec{q},k}^r(\mathbb{R}^n)} < \infty.
\end{align*}
By the completeness of the mixed Dunkl--Lebesgue spaces, $\sum\limits_{j=1}^\infty f_j\chi_{B(0,R)}$ converges in $L_{\mathrm{rad},k}^{\vec{p}} L_{\mathrm{ang},k}^{\vec{q}}(\mathbb{R}^n)$ to a measurable function $g \equiv \sum\limits_{j=1}^\infty f_j$. Using the triangle inequality for partial sums and then passing to the limit yields
\begin{align*}
\|g\|_{\dot{M}_{\vec{p},\vec{q},k}^r(\mathbb{R}^n)}
&= \sup_{R>0} |B(0,R)|_k^{\frac{1}{r} - \frac{1}{N_k}\sum\limits_{i=1}^m \frac{N_{k^{(i)}}}{p_i}} \left\| \sum\limits_{j=1}^\infty f_j \chi_{B(0,R)} \right\|_{L_{\mathrm{rad},k}^{\vec{p}} L_{\mathrm{ang},k}^{\vec{q}}} \\
&\le \sum\limits_{j=1}^\infty \sup_{R>0} |B(0,R)|_k^{\frac{1}{r} - \frac{1}{N_k}\sum\limits_{i=1}^m \frac{N_{k^{(i)}}}{p_i}} \|f_j\chi_{B(0,R)}\|_{L_{\mathrm{rad},k}^{\vec{p}} L_{\mathrm{ang},k}^{\vec{q}}} \\
&= \sum_{j=1}^\infty \|f_j\|_{\dot{M}_{\vec{p},\vec{q},k}^r(\mathbb{R}^n)} < \infty.
\end{align*}
This shows that $g \in \dot{M}_{\vec{p},\vec{q},k}^r(\mathbb{R}^n)$. Finally, to prove that the partial sums converge to $g$ in the Morrey norm, we estimate the tail of the series. For any integer $N \ge 1$,
\begin{align*}
\left\| g - \sum_{j=1}^N f_j \right\|_{\dot{M}_{\vec{p},\vec{q},k}^r(\mathbb{R}^n)}
&= \left\| \sum_{j=N+1}^\infty f_j \right\|_{\dot{M}_{\vec{p},\vec{q},k}^r(\mathbb{R}^n)} \\
&\le \sum_{j=N+1}^\infty \|f_j\|_{\dot{M}_{\vec{p},\vec{q},k}^r(\mathbb{R}^n)}.
\end{align*}
Since $\sum\limits_{j=1}^\infty \|f_j\|_{\dot{M}_{\vec{p},\vec{q},k}^r(\mathbb{R}^n)} < \infty$, the right-hand side tends to $0$ as $N \to \infty$. Thus, $\sum\limits_{j=1}^N f_j \to g$ in $\dot{M}_{\vec{p},\vec{q},k}^r(\mathbb{R}^n)$, confirming that the space is complete.
\end{remark}

Just as with mixed Lebesgue spaces, the mixed Morrey norm satisfies a precise dilation relation. Utilizing the scaling identity derived in the previous section, for any $t > 0$, we have
\begin{equation}
\|f(t\cdot)\|_{\dot{M}_{\vec{p},\vec{q},k}^r(\mathbb{R}^n)} = t^{-\frac{N_k}{r}} \|f\|_{\dot{M}_{\vec{p},\vec{q},k}^r(\mathbb{R}^n)}. \label{eq:dilation}
\end{equation}

Next, we establish the embedding properties of these spaces.

\begin{proposition}\label{prop:embedding}
Let $1 < \vec{p}_1 \le \vec{p}_2 <\infty$ and $1 < \vec{q}_1 \le \vec{q}_2 < \infty$. Assume $r \in (0, \infty)$ satisfies the exponent constraint $\frac{1}{r} \le \frac{1}{N_k}\sum\limits_{i=1}^m \frac{N_{k^{(i)}}}{p_{2i}}$. Then,
\[
\dot{M}_{\vec{p}_2,\vec{q}_2,k}^r(\mathbb{R}^n) \subset \dot{M}_{\vec{p}_1,\vec{q}_1,k}^r(\mathbb{R}^n).
\]
\end{proposition}

\begin{proof}
To obtain this inclusion, it suffices to show that for all measurable functions $f$ and all balls $B(0,R)$,
\begin{equation*}
|B(0,R)|_k^{\frac{1}{r} - \frac{1}{N_k}\sum\limits_{i=1}^m \frac{N_{k^{(i)}}}{p_{1i}}} \|f\chi_{B(0,R)}\|_{L_{\mathrm{rad},k}^{\vec{p}_1} L_{\mathrm{ang},k}^{\vec{q}_1}}
\lesssim  |B(0,R)|_k^{\frac{1}{r} - \frac{1}{N_k}\sum\limits_{i=1}^m \frac{N_{k^{(i)}}}{p_{2i}}} \|f\chi_{B(0,R)}\|_{L_{\mathrm{rad},k}^{\vec{p}_2} L_{\mathrm{ang},k}^{\vec{q}_2}}.
\end{equation*}
Taking the supremum over all $R > 0$ directly yields $$\|f\|_{\dot{M}_{\vec{p}_1,\vec{q}_1,k}^r} \lesssim \|f\|_{\dot{M}_{\vec{p}_2,\vec{q}_2,k}^r}.$$

Let $F \ge 0$ be a generic non-negative function. Since $\operatorname{supp}(f\chi_{B(0,R)}) \subset \prod\limits_{i=1}^m (0, R] \times \mathbb{S}^{n_i-1}$, we iteratively apply H\"older's inequality, assuming $q_{2i} \ge q_{1i}$,
\begin{align*}
\left( \int_{\mathbb{S}^{n_i-1}} F(\rho_i \omega_i)^{q_{1i}} \,\mathrm{d}\sigma_{k^{(i)}}(\omega_i) \right)^{\frac{1}{q_{1i}}}
&\le \left( \int_{\mathbb{S}^{n_i-1}} F(\rho_i \omega_i)^{q_{2i}} \,\mathrm{d}\sigma_{k^{(i)}}(\omega_i) \right)^{\frac{1}{q_{2i}}} \left( \int_{\mathbb{S}^{n_i-1}} 1 \,\mathrm{d}\sigma_{k^{(i)}}(\omega_i) \right)^{\frac{1}{q_{1i}} - \frac{1}{q_{2i}}} \\
&= \omega_{n_i, k^{(i)}}^{\frac{1}{q_{1i}} - \frac{1}{q_{2i}}} \left( \int_{\mathbb{S}^{n_i-1}} F(\rho_i \omega_i)^{q_{2i}} \,\mathrm{d}\sigma_{k^{(i)}}(\omega_i) \right)^{\frac{1}{q_{2i}}}.
\end{align*}
Applying H\"older's inequality to the radial integral on $(0, R]$ (for $p_{2i} \ge p_{1i}$) gives
\begin{align*}
&\left( \int_0^R \left( \int_{\mathbb{S}^{n_i-1}} F(\rho_i \omega_i)^{q_{1i}} \,\mathrm{d}\sigma_{k^{(i)}}(\omega_i) \right)^{\frac{p_{1i}}{q_{1i}}} \rho_i^{N_{k^{(i)}}-1} \,\mathrm{d}\rho_i \right)^{\frac{1}{p_{1i}}} \\
&\le \omega_{n_i, k^{(i)}}^{\frac{1}{q_{1i}} - \frac{1}{q_{2i}}} \left( \int_0^R \left( \int_{\mathbb{S}^{n_i-1}} F(\rho_i \omega_i)^{q_{2i}} \,\mathrm{d}\sigma_{k^{(i)}}(\omega_i) \right)^{\frac{p_{1i}}{q_{2i}}} \rho_i^{N_{k^{(i)}}-1} \,\mathrm{d}\rho_i \right)^{\frac{1}{p_{1i}}} \\
&\le \omega_{n_i, k^{(i)}}^{\frac{1}{q_{1i}} - \frac{1}{q_{2i}}} \left( \int_0^R \left( \int_{\mathbb{S}^{n_i-1}} F(\rho_i \omega_i)^{q_{2i}} \,\mathrm{d}\sigma_{k^{(i)}}(\omega_i) \right)^{\frac{p_{2i}}{q_{2i}}} \rho_i^{N_{k^{(i)}}-1} \,\mathrm{d}\rho_i \right)^{\frac{1}{p_{2i}}} \left( \int_0^R \rho_i^{N_{k^{(i)}}-1} \,\mathrm{d}\rho_i \right)^{\frac{1}{p_{1i}} - \frac{1}{p_{2i}}} \\
&=\omega_{n_i, k^{(i)}}^{\frac{1}{q_{1i}} - \frac{1}{q_{2i}}} (N_{k^{(i)}})^{\frac{1}{p_{2i}} - \frac{1}{p_{1i}}} R^{N_{k^{(i)}}\left(\frac{1}{p_{1i}} - \frac{1}{p_{2i}}\right)} \left( \int_0^R \left( \int_{\mathbb{S}^{n_i-1}} F(\rho_i \omega_i)^{q_{2i}} \,\mathrm{d}\sigma_{k^{(i)}}(\omega_i) \right)^{\frac{p_{2i}}{q_{2i}}} \rho_i^{N_{k^{(i)}}-1} \,\mathrm{d}\rho_i \right)^{\frac{1}{p_{2i}}}.
\end{align*}
Iterating over $i = 1, \dots, m$ yields
\[
\|f\chi_{B(0,R)}\|_{L_{\mathrm{rad},k}^{\vec{p}_1} L_{\mathrm{ang},k}^{\vec{q}_1}} \le \prod_{i=1}^m \omega_{n_i, k^{(i)}}^{\frac{1}{q_{1i}} - \frac{1}{q_{2i}}} (N_{k^{(i)}})^{\frac{1}{p_{2i}} - \frac{1}{p_{1i}}} R^{\sum\limits_{i=1}^m N_{k^{(i)}}\left(\frac{1}{p_{1i}} - \frac{1}{p_{2i}}\right)} \|f\chi_{B(0,R)}\|_{L_{\mathrm{rad},k}^{\vec{p}_2} L_{\mathrm{ang},k}^{\vec{q}_2}}.
\]

Since $|B(0,R)|_k = \nu_{n,k} R^{N_k}$, substituting $R = \nu_{n,k}^{-\frac{1}{N_k}} |B(0,R)|_k^{\frac{1}{N_k}}$ gives
\begin{align*}
R^{\sum\limits_{i=1}^m N_{k^{(i)}}\left(\frac{1}{p_{1i}} - \frac{1}{p_{2i}}\right)}
&= \nu_{n,k}^{-\frac{1}{N_k}\sum\limits_{i=1}^m N_{k^{(i)}}\left(\frac{1}{p_{1i}} - \frac{1}{p_{2i}}\right)} |B(0,R)|_k^{\frac{1}{N_k}\sum\limits_{i=1}^m \frac{N_{k^{(i)}}}{p_{1i}} - \frac{1}{N_k}\sum\limits_{i=1}^m \frac{N_{k^{(i)}}}{p_{2i}}}.
\end{align*}
Multiplying by $|B(0,R)|_k^{\frac{1}{r} - \frac{1}{N_k}\sum\limits_{i=1}^m \frac{N_{k^{(i)}}}{p_{1i}}}$, we obtain
\begin{align*}
&|B(0,R)|_k^{\frac{1}{r} - \frac{1}{N_k}\sum\limits_{i=1}^m \frac{N_{k^{(i)}}}{p_{1i}}} \|f\chi_{B(0,R)}\|_{L_{\mathrm{rad},k}^{\vec{p}_1} L_{\mathrm{ang},k}^{\vec{q}_1}} \\
&\lesssim |B(0,R)|_k^{\frac{1}{r} - \frac{1}{N_k}\sum\limits_{i=1}^m \frac{N_{k^{(i)}}}{p_{1i}}} \cdot |B(0,R)|_k^{\frac{1}{N_k}\sum\limits_{i=1}^m \frac{N_{k^{(i)}}}{p_{1i}} - \frac{1}{N_k}\sum\limits_{i=1}^m \frac{N_{k^{(i)}}}{p_{2i}}} \|f\chi_{B(0,R)}\|_{L_{\mathrm{rad},k}^{\vec{p}_2} L_{\mathrm{ang},k}^{\vec{q}_2}} \\
&=|B(0,R)|_k^{\frac{1}{r} - \frac{1}{N_k}\sum\limits_{i=1}^m \frac{N_{k^{(i)}}}{p_{2i}}} \|f\chi_{B(0,R)}\|_{L_{\mathrm{rad},k}^{\vec{p}_2} L_{\mathrm{ang},k}^{\vec{q}_2}}.
\end{align*}
\end{proof}

Let us give some fundamental examples of functions in these spaces.

\begin{example}\label{ex:radial_decay}
Let $\vec{\alpha} = (\alpha_1, \dots, \alpha_m)$ with $\alpha_i > 0$ such that $\sum\limits_{i=1}^m \alpha_i = \frac{N_k}{r}$. Then
\begin{equation}
f(x) = \prod_{i=1}^m \rho_i^{-\alpha_i} \in \dot{M}_{\vec{p},\vec{q},k}^r(\mathbb{R}^n) \label{eq:radial_function}
\end{equation}
if and only if $p_i < \frac{N_{k^{(i)}}}{\alpha_i}$ for all $i = 1, \dots, m$.
\end{example}

\begin{proof}
Using the radial dilation $\rho_i = R\tau_i$, we have
\[
\|f\chi_{B(0,R)}\|_{L_{\mathrm{rad},k}^{\vec{p}} L_{\mathrm{ang},k}^{\vec{q}}} = R^{-\sum\limits_{i=1}^m \alpha_i + \sum\limits_{i=1}^m \frac{N_{k^{(i)}}}{p_i}} \|f\chi_{B(0,1)}\|_{L_{\mathrm{rad},k}^{\vec{p}} L_{\mathrm{ang},k}^{\vec{q}}}.
\]
Let $Q_c = \{\rho_i \in (0, c): 1 \le i \le m\}$. Since $\chi_{Q_{1/\sqrt{m}}} \le \chi_{B(0,1)} \le \chi_{Q_1}$, the norm of $f\chi_{B(0,1)}$ is finite if the norm over $Q_1$ is finite. Because $f$ is purely radial, the mixed norm over the product domain $Q_1$ completely decouples into independent angular and radial evaluations:
\begin{align*}
\|f\chi_{Q_1}\|_{L_{\mathrm{rad},k}^{\vec{p}} L_{\mathrm{ang},k}^{\vec{q}}}
&=\left( \int_{\mathbb{S}^{n_m-1}} \dots \left( \int_{\mathbb{S}^{n_2-1}} \left( \int_{\mathbb{S}^{n_1-1}} 1^{q_1} \,\mathrm{d}\sigma_{k^{(1)}}(\omega_1) \right)^{\frac{q_2}{q_1}} \,\mathrm{d}\sigma_{k^{(2)}}(\omega_2) \right)^{\frac{q_3}{q_2}} \dots \,\mathrm{d}\sigma_{k^{(m)}}(\omega_m) \right)^{\frac{1}{q_m}} \\
&\quad\times\prod_{i=1}^m \left( \int_0^1 \rho_i^{-\alpha_i p_i + N_{k^{(i)}}-1} \,\mathrm{d}\rho_i \right)^{\frac{1}{p_i}} \\
&= \left( \prod_{i=1}^m \omega_{n_i, k^{(i)}}^{\frac{1}{q_i}} \right) \prod_{i=1}^m \left( \int_0^1 \rho_i^{-\alpha_i p_i + N_{k^{(i)}}-1} \,\mathrm{d}\rho_i \right)^{\frac{1}{p_i}}.
\end{align*}
This product is finite if and only if each one-dimensional integral converges at the origin, which requires $-\alpha_i p_i + N_{k^{(i)}} - 1 > -1$, or equivalently, $p_i < \frac{N_{k^{(i)}}}{\alpha_i}$ for all $i = 1, \dots, m$.

Assuming this finiteness condition holds, and utilizing $|B(0,R)|_k = \nu_{n,k} R^{N_k}$ along with the constraint $\sum\limits_{i=1}^m \alpha_i = \frac{N_k}{r}$, we obtain
\begin{align*}
\|f\|_{\dot{M}_{\vec{p},\vec{q},k}^r(\mathbb{R}^n)}
&= \sup_{R>0} |B(0,R)|_k^{\frac{1}{r} - \frac{1}{N_k}\sum\limits_{i=1}^m \frac{N_{k^{(i)}}}{p_i}} \|f\chi_{B(0,R)}\|_{L_{\mathrm{rad},k}^{\vec{p}} L_{\mathrm{ang},k}^{\vec{q}}} \\
&= \sup_{R>0} \left( \nu_{n,k} R^{N_k} \right)^{\frac{1}{r} - \frac{1}{N_k}\sum\limits_{i=1}^m \frac{N_{k^{(i)}}}{p_i}} R^{-\sum\limits_{i=1}^m \alpha_i + \sum\limits_{i=1}^m \frac{N_{k^{(i)}}}{p_i}} \|f\chi_{B(0,1)}\|_{L_{\mathrm{rad},k}^{\vec{p}} L_{\mathrm{ang},k}^{\vec{q}}} \\
&= \sup_{R>0} \nu_{n,k}^{\frac{1}{r} - \frac{1}{N_k}\sum\limits_{i=1}^m \frac{N_{k^{(i)}}}{p_i}} R^{\frac{N_k}{r} - \sum\limits_{i=1}^m \alpha_i} \|f\chi_{B(0,1)}\|_{L_{\mathrm{rad},k}^{\vec{p}} L_{\mathrm{ang},k}^{\vec{q}}} \\
&= \nu_{n,k}^{\frac{1}{r} - \frac{1}{N_k}\sum\limits_{i=1}^m \frac{N_{k^{(i)}}}{p_i}} \|f\chi_{B(0,1)}\|_{L_{\mathrm{rad},k}^{\vec{p}} L_{\mathrm{ang},k}^{\vec{q}}} \\
&< \infty.
\end{align*}
\end{proof}

\begin{remark}\label{rem:necessity}
In Example \ref{ex:radial_decay}, the condition $\sum\limits_{i=1}^m \alpha_i = \frac{N_k}{r}$ is an absolute necessity for $f$ to be a non-trivial member of $\dot{M}_{\vec{p},\vec{q},k}^r(\mathbb{R}^n)$. This is easily verified by the dilation property \eqref{eq:dilation}. Note that $f(t x) = t^{-\sum\limits_{i=1}^m \alpha_i} f(x)$. If $f \in \dot{M}_{\vec{p},\vec{q},k}^r(\mathbb{R}^n)$ and $f \not\equiv 0$, applying \eqref{eq:dilation} gives
\[
t^{-\frac{N_k}{r}} \|f\|_{\dot{M}_{\vec{p},\vec{q},k}^r(\mathbb{R}^n)} = \|f(t\cdot)\|_{\dot{M}_{\vec{p},\vec{q},k}^r(\mathbb{R}^n)} = t^{-\sum\limits_{i=1}^m \alpha_i} \|f\|_{\dot{M}_{\vec{p},\vec{q},k}^r(\mathbb{R}^n)}.
\]
For this to hold for all $t > 0$, the exponents must be strictly equal.
\end{remark}

\subsection{Operators and Commutators in the Dunkl setting}\label{subsec:operators}

Recall that, in the Dunkl setting, the generalized Hausdorff operator \cite{LZZ2026} is defined by
\[
H_{\Phi,s,k}f(x)
=
\int_{\mathbb{R}^n}
\frac{\Phi(y)}{|y|^{N_k}}
f(s(|y|)x)\,\mathrm{d}\mu_k(y).
\]

The weighted Hardy--Littlewood operator considered in this paper is a special case of \(H_{\Phi,s,k}\). Let
\[
\Phi(y)=\frac{1}{\omega_{n,k}}|y|\varphi(|y|)
\chi_{\{|y|\leq1\}}(y),
\qquad s(t)=t.
\]
Let $\varphi$ be a non-negative measurable function on $(0,1]$. Utilizing the polar coordinate, we have
\[
\begin{aligned}
H_{\Phi,s,k}f(x)
&=\int_{\mathbb{R}^n}\frac{\Phi(y)}{|y|^{N_k}}f(|y|x)\,\mathrm{d}\mu_k(y)\\
&=\frac{1}{\omega_{n,k}}\int_{\mathbb{S}^{n-1}}\int_0^1\frac{\rho\varphi(\rho)}{\rho^{N_k}}f(\rho x)\rho^{N_k-1}\,\mathrm{d}\rho\,\mathrm{d}\sigma_k(\theta)\\
&=\int_0^1 f(\rho x)\varphi(\rho)\,\mathrm{d}\rho.
\end{aligned}
\]
Hence, with $\rho=t$, the weighted Hardy--Littlewood operator takes the simple form
\[
\mathcal{H}_{\varphi}f(x)
=\int_0^1 f(tx)\varphi(t)\,\mathrm{d}t.
\]
Consequently, for a locally integrable function $b$, the corresponding commutator $\mathcal{H}_{\varphi}^{b}$ in the Dunkl setting is defined by
$$
\mathcal{H}_{\varphi}^{b} f(x) = b(x)\mathcal{H}_\varphi f(x) - \mathcal{H}_\varphi(bf)(x) = \int_0^1 (b(x)-b(tx)) f(tx)\varphi(t)\,\mathrm{d}t.
$$
\section{Sharp constants for $\mathcal{H}_{\varphi}$ on $\dot{M}_{\vec p, \vec q, k}^r(\mathbb{R}^n)$ and $\dot B_k^{\vec p, \vec q,\lambda}(\mathbb R^n)$}\label{sec3}

This section is devoted to investigating the boundedness of weighted Hardy--Littlewood averages on mixed-norm radial-angular Dunkl central Morrey space $\dot{M}_{\vec{p},\vec{q},k}^r(\mathbb{R}^n)$ and mixed-norm radial-angular Dunkl $\lambda$-central Morrey space $\dot{B}^{\vec p, \vec q,\lambda}_k(\mathbb{R}^n)$.
\begin{theorem}\label{th1}
Let $0<r<\infty$, $1<\vec p, \vec q<\infty$, and
$\sum\limits_{i=1}^m\frac{N_{k^{(i)}}}{p_i}\geqslant \frac{N_k}{r}.$
Then $\mathcal{H}_{\varphi}$ is bounded on $\dot{M}_{\vec p, \vec q, k}^r(\mathbb{R}^n)$ if and only if
\begin{equation}
\int_0^1 t^{-\frac{N_k}{r}}\varphi(t)\,\mathrm{d}t<\infty.
\end{equation}
Moreover,
\[
\|\mathcal{H}_{\varphi}\|_{\dot{M}_{\vec p, \vec q, k}^r(\mathbb{R}^n)\to
\dot{M}_{\vec p, \vec q, k}^r(\mathbb{R}^n)}
=\int_0^1 t^{-\frac{N_k}{r}}\varphi(t)\,\mathrm{d}t.
\]
\end{theorem}

\begin{theorem}\label{2}
Suppose $1<\vec p, \vec q<\infty$ and $\lambda>-\frac1{N_k}\sum_{i=1}^{m}\frac{N_{k^{(i)}}}{p_i}.$
Then $\mathcal{H}_{\varphi}$ is bounded on $\dot B_k^{\vec p, \vec q,\lambda}(\mathbb R^n)$ if and only if
\begin{equation}
\int_0^1 t^{N_k\lambda}\varphi(t)\,\mathrm{d}t<\infty.\label{3.2}
\end{equation}
Furthermore,
\[
\left\|\mathcal{H}_{\varphi}\right\|_{\dot B_k^{\vec p, \vec q,\lambda}(\mathbb R^n)
\to \dot B_k^{\vec p, \vec q,\lambda}(\mathbb R^n)}
=
\int_0^1 t^{N_k\lambda}\varphi(t)\,\mathrm{d}t .
\]
\end{theorem}

\begin{proof}[Proof of Theorem 3.1]
For any ball $B(0,R)$, $R>0$, by Minkowski's inequality, we have
\begin{align*}
&|B(0,R)|_k^{\frac1r-\frac1{N_k}\left(\sum\limits_{i=1}^m\frac{N_{k^{(i)}}}{p_i}\right)}\|\mathcal{H}_{\varphi}f\chi_{B(0,R)}\|_{L_{\mathrm{rad}, k}^{\vec p} L_{\mathrm{ang}, k}^{\vec q}(\mathbb{R}^n)}\\
&\leqslant\int_0^1\varphi(t)|B(0,R)|_k^{\frac1r-\frac1{N_k}\left(\sum\limits_{i=1}^m\frac{N_{k^{(i)}}}{p_i}\right)}\|f(t\cdot)\chi_{B(0,R)}(\cdot)\|_{L_{\mathrm{rad}, k}^{\vec p} L_{\mathrm{ang}, k}^{\vec q}(\mathbb{R}^n)}\,\mathrm{d}t\\
&=\int_0^1\varphi(t)|B(0,R)|_k^{\frac1r-\frac1{N_k}\left(\sum\limits_{i=1}^m\frac{N_{k^{(i)}}}{p_i}\right)}t^{-\sum\limits_{i=1}^m\frac{N_{k^{(i)}}}{p_i}}\|f\chi_{B(0,tR)}\|_{L_{\mathrm{rad}, k}^{\vec p} L_{\mathrm{ang}, k}^{\vec q}(\mathbb{R}^n)}\,\mathrm{d}t\\
&=\int_0^1t^{-\frac{N_k}{r}}\varphi(t)|B(0,tR)|_k^{\frac1r-\frac1{N_k}\left(\sum\limits_{i=1}^m\frac{N_{k^{(i)}}}{p_i}\right)}\|f\chi_{B(0,tR)}\|_{L_{\mathrm{rad}, k}^{\vec p} L_{\mathrm{ang}, k}^{\vec q}(\mathbb{R}^n)}\,\mathrm{d}t\\
&\leqslant\|f\|_{\dot{M}_{\vec p, \vec q, k}^r(\mathbb{R}^n)}\int_0^1 t^{-\frac{N_k}{r}}\varphi(t)\,\mathrm{d}t.
\end{align*}
Taking the supremum over all such balls $B(0,R)$ with $R>0$, we obtain
\begin{equation*}
\|\mathcal{H}_{\varphi}f\|_{\dot{M}_{\vec p, \vec q, k}^r(\mathbb{R}^n)}
\leqslant \int_0^1 t^{-\frac{N_k}{r}}\varphi(t)\,\mathrm{d}t\|f\|_{\dot{M}_{\vec p, \vec q, k}^r(\mathbb{R}^n)}.
\end{equation*}

Next we prove that the constant $\int_0^1 t^{-\frac{N_k}{r}}\varphi(t)\,\mathrm{d}t$ is sharp. When
\[
\sum_{i=1}^m\frac{N_{k^{(i)}}}{p_i}>\frac{N_k}{r},
\]
we choose a purely radial function
\[
f_0(x) = f_0(\rho_1\omega_1, \dots, \rho_m\omega_m)
=
\prod_{i=1}^m \rho_i^{-\frac{N_{k^{(i)}}}{s_i}},
\]
where $p_i<s_i<\infty$ and $\sum\limits_{i=1}^m\frac{N_{k^{(i)}}}{s_i}=\frac{N_k}{r}$.
Such $s_i$ exist since $\sum_{i=1}^m\frac{N_{k^{(i)}}}{p_i}>\frac{N_k}{r}$.
Moreover, since $p_i<s_i$, we have
\[
-\frac{N_{k^{(i)}} p_i}{s_i} + N_{k^{(i)}} - 1 > -1,
\]
and hence
\[
\int_{0}^{1}
\rho_i^{-\frac{N_{k^{(i)}} p_i}{s_i}}
\rho_i^{N_{k^{(i)}}-1}\,\mathrm{d}\rho_i
<\infty.
\]
Thus
\[
\|f_0\chi_{B(0,1)}\|_{L_{\mathrm{rad}, k}^{\vec p} L_{\mathrm{ang}, k}^{\vec q}(\mathbb{R}^n)}<\infty.
\]
By making the changes of variables \(\rho_i=R\tau_i\), \(i=1,\ldots,m\), we obtain
\begin{align*}
&\left\|f_0\chi_{B(0,R)}\right\|_{L_{\mathrm{rad},k}^{\vec p}L_{\mathrm{ang},k}^{\vec q}(\mathbb R^n)}\\
&=\left\|\prod_{i=1}^m\rho_i^{-\frac{N_{k^{(i)}}}{s_i}}\chi_{\left\{\sum\limits_{i=1}^m\rho_i^2<R^2\right\}}\right\|_{L_{\mathrm{rad},k}^{\vec p}L_{\mathrm{ang},k}^{\vec q}(\mathbb R^n)}\\
&=R^{-\sum\limits_{i=1}^m\frac{N_{k^{(i)}}}{s_i}}R^{\sum\limits_{i=1}^m\frac{N_{k^{(i)}}}{p_i}}\left\|\prod_{i=1}^m\tau_i^{-\frac{N_{k^{(i)}}}{s_i}}\chi_{\left\{\sum\limits_{i=1}^m\tau_i^2<1\right\}}\right\|_{L_{\mathrm{rad},k}^{\vec p}L_{\mathrm{ang},k}^{\vec q}(\mathbb R^n)}\\
&=R^{\sum\limits_{i=1}^m\frac{N_{k^{(i)}}}{p_i}-\sum\limits_{i=1}^m\frac{N_{k^{(i)}}}{s_i}}\left\|f_0\chi_{B(0,1)}\right\|_{L_{\mathrm{rad},k}^{\vec p}
L_{\mathrm{ang},k}^{\vec q}(\mathbb R^n)}\\
&=R^{\sum\limits_{i=1}^m\frac{N_{k^{(i)}}}{p_i}-\frac{N_k}{r}}\left\|f_0\chi_{B(0,1)}\right\|_{L_{\mathrm{rad},k}^{\vec p}L_{\mathrm{ang},k}^{\vec q}(\mathbb R^n)},
\end{align*}
where we have used
$\sum\limits_{i=1}^m\frac{N_{k^{(i)}}}{s_i}=\frac{N_k}{r}$. Therefore,
\begin{align*}
&|B(0,R)|_k^{\frac1r-\frac1{N_k}\sum\limits_{i=1}^m\frac{N_{k^{(i)}}}{p_i}}\left\|f_0\chi_{B(0,R)}\right\|_{L_{\mathrm{rad},k}^{\vec p}L_{\mathrm{ang},k}^{\vec q}(\mathbb R^n)}\\
&=\left(\nu_{n,k}R^{N_k}\right)^{\frac1r-\frac1{N_k}\sum\limits_{i=1}^m\frac{N_{k^{(i)}}}{p_i}}R^{\sum\limits_{i=1}^m\frac{N_{k^{(i)}}}{p_i}-\frac{N_k}{r}}\left\|f_0\chi_{B(0,1)}\right\|_{L_{\mathrm{rad},k}^{\vec p}L_{\mathrm{ang},k}^{\vec q}(\mathbb R^n)}\\
&=\nu_{n,k}^{\frac1r-\frac1{N_k}\sum\limits_{i=1}^m\frac{N_{k^{(i)}}}{p_i}}R^{\frac{N_k}{r}-\sum\limits_{i=1}^m\frac{N_{k^{(i)}}}{p_i}+\sum\limits_{i=1}^m\frac{N_{k^{(i)}}}{p_i}
-\frac{N_k}{r}}\left\|f_0\chi_{B(0,1)}\right\|_{L_{\mathrm{rad},k}^{\vec p}L_{\mathrm{ang},k}^{\vec q}(\mathbb R^n)}\\
&=\nu_{n,k}^{\frac1r-\frac1{N_k}\sum\limits_{i=1}^m\frac{N_{k^{(i)}}}{p_i}}
\left\|f_0\chi_{B(0,1)}\right\|_{L_{\mathrm{rad},k}^{\vec p}L_{\mathrm{ang},k}^{\vec q}(\mathbb R^n)}.
\end{align*}

Hence
\[
f_0\in \dot{M}_{\vec p, \vec q, k}^r(\mathbb{R}^n).
\]
By a direct computation,
\begin{align*}
\mathcal{H}_{\varphi}f_0(x)
&=\int_0^1\prod_{i=1}^m (t\rho_i)^{-\frac{N_{k^{(i)}}}{s_i}}\varphi(t)\,\mathrm{d}t\\
&=\prod_{i=1}^m \rho_i^{-\frac{N_{k^{(i)}}}{s_i}}\int_0^1t^{-\sum_{i=1}^m\frac{N_{k^{(i)}}}{s_i}}\varphi(t)\,\mathrm{d}t\\
&=f_0(x)\int_0^1 t^{-\frac{N_k}{r}}\varphi(t)\,\mathrm{d}t.
\end{align*}
As a consequence,
\[
\|\mathcal{H}_{\varphi}f_0\|_{\dot{M}_{\vec p, \vec q, k}^r(\mathbb{R}^n)}
=\int_0^1 t^{-\frac{N_k}{r}}\varphi(t)\,\mathrm{d}t\|f_0\|_{\dot{M}_{\vec p, \vec q, k}^r(\mathbb{R}^n)}.
\]
Thus
\[
\|\mathcal{H}_{\varphi}\|_{\dot{M}_{\vec p, \vec q, k}^r(\mathbb{R}^n)\to
\dot{M}_{\vec p, \vec q, k}^r(\mathbb{R}^n)}
\geqslant \int_0^1 t^{-\frac{N_k}{r}}\varphi(t)\,\mathrm{d}t.
\]

When
\[
\sum_{i=1}^m\frac{N_{k^{(i)}}}{p_i}=\frac{N_k}{r},
\]
then $\dot{M}_{\vec p, \vec q, k}^r(\mathbb{R}^n)$ is just the mixed-norm radial-angular Dunkl Lebesgue space
$L_{\mathrm{rad}, k}^{\vec p} L_{\mathrm{ang}, k}^{\vec q}(\mathbb{R}^n)$.
In this case we take the purely radial function
\[
f_{\varepsilon}(x) = f_{\varepsilon}(\rho_1\omega_1, \dots, \rho_m\omega_m)
=\prod_{i=1}^m\rho_i^{-\frac{N_{k^{(i)}}}{p_i}-\varepsilon}\chi_{\{\rho_i>1\}},
\]
where $\varepsilon>0$ is sufficiently small.

By a routine calculation, we get
\begin{align*}
\|f_{\varepsilon}\|_{L_{\mathrm{rad}, k}^{\vec p} L_{\mathrm{ang}, k}^{\vec q}(\mathbb{R}^n)}
&=\left( \int_{\mathbb{S}^{n_m-1}} \cdots \left( \int_{\mathbb{S}^{n_1-1}} 1 \,\mathrm{d}\sigma_{k^{(1)}}(\omega_1) \right)^{\frac{q_2}{q_1}} \cdots \,\mathrm{d}\sigma_{k^{(m)}}(\omega_m) \right)^{\frac{1}{q_m}}\\
&\quad\times\prod_{i=1}^m\left(\int_{\rho_i>1}\rho_i^{-\left(\frac{N_{k^{(i)}}}{p_i}+\varepsilon\right)p_i}\rho_i^{N_{k^{(i)}}-1}\,\mathrm{d}\rho_i\right)^{\frac1{p_i}}\\
&=\left( \prod_{i=1}^m \omega_{n_i, k^{(i)}}^{\frac{1}{q_i}} \right)\prod_{i=1}^m\left(\int_{1}^{\infty}\rho_i^{-1-p_i\varepsilon}\,\mathrm{d}\rho_i\right)^{\frac1{p_i}}\\
&=\left( \prod_{i=1}^m \omega_{n_i, k^{(i)}}^{\frac{1}{q_i}} \right)\prod_{i=1}^m(p_i\varepsilon)^{-\frac1{p_i}}.
\end{align*}
Inserting $f_{\varepsilon}$ into $\mathcal{H}_{\varphi}f$, we get
\begin{align*}
\mathcal{H}_{\varphi}f_{\varepsilon}(x)
&=\prod_{i=1}^m\rho_i^{-\frac{N_{k^{(i)}}}{p_i}-\varepsilon}
\int_{\max\limits_{1\leq v\leq m}\left\{\frac1{\rho_v}\right\}}^{1}
t^{-\sum\limits_{v=1}^m\frac{N_{k^{(v)}}}{p_v}-m\varepsilon}
\varphi(t)\,\mathrm{d}t.
\end{align*}
Therefore, for sufficiently small $\varepsilon>0$, we have
\begin{align*}
&\|\mathcal{H}_{\varphi}f_{\varepsilon}\|_{L_{\mathrm{rad}, k}^{\vec p} L_{\mathrm{ang}, k}^{\vec q}(\mathbb{R}^n)}\\
&\geqslant\int_{\varepsilon}^{1}t^{-\sum_{i=1}^m\frac{N_{k^{(i)}}}{p_i}-m\varepsilon}\varphi(t)\,\mathrm{d}t
\left\|\prod_{i=1}^m\rho_i^{-\frac{N_{k^{(i)}}}{p_i}-\varepsilon}\chi_{\{\rho_i>\frac1\varepsilon\}}\right\|_{L_{\mathrm{rad}, k}^{\vec p} L_{\mathrm{ang}, k}^{\vec q}(\mathbb{R}^n)}\\
&=\int_{\varepsilon}^{1}t^{-\sum_{i=1}^m\frac{N_{k^{(i)}}}{p_i}-m\varepsilon}\varphi(t)\,\mathrm{d}t
\left( \prod_{i=1}^m \omega_{n_i, k^{(i)}}^{\frac{1}{q_i}} \right) \prod_{i=1}^m
\left(\int_{\rho_i>\frac1\varepsilon}\rho_i^{-\left(\frac{N_{k^{(i)}}}{p_i}+\varepsilon\right)p_i}\rho_i^{N_{k^{(i)}}-1}\,\mathrm{d}\rho_i\right)^{\frac1{p_i}}\\
&=\int_{\varepsilon}^{1}t^{-\sum_{i=1}^m\frac{N_{k^{(i)}}}{p_i}-m\varepsilon}\varphi(t)\,\mathrm{d}t
\left( \prod_{i=1}^m \omega_{n_i, k^{(i)}}^{\frac{1}{q_i}} \right)\prod_{i=1}^m\left(\int_{\rho_i>\frac1\varepsilon}\rho_i^{-1-p_i\varepsilon}\,\mathrm{d}\rho_i\right)^{\frac1{p_i}}\\
&=\varepsilon^{m\varepsilon}\left( \prod_{i=1}^m \omega_{n_i, k^{(i)}}^{\frac{1}{q_i}} \right)\prod_{i=1}^m(p_i\varepsilon)^{-\frac1{p_i}}
\int_{\varepsilon}^{1}t^{-\sum_{i=1}^m\frac{N_{k^{(i)}}}{p_i}-m\varepsilon}\varphi(t)\,\mathrm{d}t\\
&=\varepsilon^{m\varepsilon}\|f_{\varepsilon}\|_{L_{\mathrm{rad}, k}^{\vec p} L_{\mathrm{ang}, k}^{\vec q}(\mathbb{R}^n)}\int_{\varepsilon}^{1}t^{-\sum_{i=1}^m\frac{N_{k^{(i)}}}{p_i}-m\varepsilon}\varphi(t)\,\mathrm{d}t.
\end{align*}
Consequently,
\[
\frac{
\|\mathcal{H}_{\varphi}f_{\varepsilon}\|_{L_{\mathrm{rad}, k}^{\vec p} L_{\mathrm{ang}, k}^{\vec q}(\mathbb{R}^n)}
}{
\|f_{\varepsilon}\|_{L_{\mathrm{rad}, k}^{\vec p} L_{\mathrm{ang}, k}^{\vec q}(\mathbb{R}^n)}
}
\geqslant
\varepsilon^{m\varepsilon}
\int_{\varepsilon}^{1}
t^{-\sum_{i=1}^m\frac{N_{k^{(i)}}}{p_i}-m\varepsilon}
\varphi(t)\,\mathrm{d}t.
\]
Letting $\varepsilon\to0^+$, and using $\lim\limits_{\varepsilon\to0^+}\varepsilon^{m\varepsilon}=1$, we obtain
\[
\|\mathcal{H}_{\varphi}\|_{L_{\mathrm{rad}, k}^{\vec p} L_{\mathrm{ang}, k}^{\vec q}(\mathbb{R}^n)\to L_{\mathrm{rad}, k}^{\vec p} L_{\mathrm{ang}, k}^{\vec q}(\mathbb{R}^n)}
\geqslant\int_0^1t^{-\sum\limits_{i=1}^m\frac{N_{k^{(i)}}}{p_i}}\varphi(t)\,\mathrm{d}t
=\int_0^1 t^{-\frac{N_k}{r}}\varphi(t)\,\mathrm{d}t.
\]
Since
\[
\dot{M}_{\vec p, \vec q, k}^r(\mathbb{R}^n)=L_{\mathrm{rad}, k}^{\vec p} L_{\mathrm{ang}, k}^{\vec q}(\mathbb{R}^n)
\]
in the case $\sum\limits_{i=1}^m\frac{N_{k^{(i)}}}{p_i}=\frac{N_k}{r}$, we also have
\[
\|\mathcal{H}_{\varphi}\|_{\dot{M}_{\vec p, \vec q, k}^r(\mathbb{R}^n)\to
\dot{M}_{\vec p, \vec q, k}^r(\mathbb{R}^n)}
\geqslant \int_0^1 t^{-\frac{N_k}{r}}\varphi(t)\,\mathrm{d}t.
\]

Combining the above two cases, we obtain
\[
\|\mathcal{H}_{\varphi}\|_{\dot{M}_{\vec p, \vec q, k}^r(\mathbb{R}^n)\to
\dot{M}_{\vec p, \vec q, k}^r(\mathbb{R}^n)}
=\int_0^1 t^{-\frac{N_k}{r}}\varphi(t)\,\mathrm{d}t.
\]
The proof of Theorem \ref{th1} is completed.
\end{proof}

\begin{proof}[Proof of Theorem 3.2]
Suppose that (\ref{3.2}) holds. For $R>0$, using Minkowski's inequality and the changes of variables $\tau_i=t\rho_i$, we have
\begin{align*}
\frac{\left\|\mathcal{H}_{\varphi} f\,\chi_{B(0,R)}\right\|_{L_{\mathrm{rad}, k}^{\vec p} L_{\mathrm{ang}, k}^{\vec q}(\mathbb R^n)}}{
|B(0,R)|_k^\lambda\left\|\chi_{B(0,R)}\right\|_{L_{\mathrm{rad}, k}^{\vec p} L_{\mathrm{ang}, k}^{\vec q}(\mathbb R^n)}}
\leq\int_0^1 \varphi(t)\frac{\left\|f(t\cdot)\chi_{B(0,R)}\right\|_{L_{\mathrm{rad}, k}^{\vec p} L_{\mathrm{ang}, k}^{\vec q}(\mathbb R^n)}
}{|B(0,R)|_k^\lambda\left\|\chi_{B(0,R)}\right\|_{L_{\mathrm{rad}, k}^{\vec p} L_{\mathrm{ang}, k}^{\vec q}(\mathbb R^n)}}\,\mathrm{d}t .
\end{align*}
Since
\[
\left\|f(t\cdot)\chi_{B(0,R)}\right\|_{L_{\mathrm{rad}, k}^{\vec p} L_{\mathrm{ang}, k}^{\vec q}(\mathbb R^n)}
=t^{-\sum\limits_{i=1}^m\frac{N_{k^{(i)}}}{p_i}}\left\|f\,\chi_{B(0,tR)}\right\|_{L_{\mathrm{rad}, k}^{\vec p} L_{\mathrm{ang}, k}^{\vec q}(\mathbb R^n)},
\]
and
\[
|B(0,tR)|_k=t^{N_k}|B(0,R)|_k,
\qquad\left\|\chi_{B(0,tR)}\right\|_{L_{\mathrm{rad}, k}^{\vec p} L_{\mathrm{ang}, k}^{\vec q}(\mathbb R^n)}
=t^{\sum\limits_{i=1}^m\frac{N_{k^{(i)}}}{p_i}}\left\|\chi_{B(0,R)}\right\|_{L_{\mathrm{rad}, k}^{\vec p} L_{\mathrm{ang}, k}^{\vec q}(\mathbb R^n)},
\]
we obtain
\begin{align*}
&\frac{\left\|\mathcal{H}_{\varphi} f\,\chi_{B(0,R)}\right\|_{L_{\mathrm{rad}, k}^{\vec p} L_{\mathrm{ang}, k}^{\vec q}(\mathbb R^n)}
}{|B(0,R)|_k^\lambda\left\|\chi_{B(0,R)}\right\|_{L_{\mathrm{rad}, k}^{\vec p} L_{\mathrm{ang}, k}^{\vec q}(\mathbb R^n)}} \\
&\leq\int_0^1 t^{N_k\lambda}\varphi(t)\frac{\left\|f\,\chi_{B(0,tR)}\right\|_{L_{\mathrm{rad}, k}^{\vec p} L_{\mathrm{ang}, k}^{\vec q}(\mathbb R^n)}}{
|B(0,tR)|_k^\lambda\left\|\chi_{B(0,tR)}\right\|_{L_{\mathrm{rad}, k}^{\vec p} L_{\mathrm{ang}, k}^{\vec q}(\mathbb R^n)}}\,\mathrm{d}t \\
&\leq\left\|f\right\|_{\dot B_k^{\vec p, \vec q,\lambda}(\mathbb R^n)}\int_0^1 t^{N_k\lambda}\varphi(t)\,\mathrm{d}t .
\end{align*}
Taking the supremum over all $R>0$, we get
\[
\left\|\mathcal{H}_{\varphi} f\right\|_{\dot B_k^{\vec p, \vec q,\lambda}(\mathbb R^n)}
\leq
\left\|f\right\|_{\dot B_k^{\vec p, \vec q,\lambda}(\mathbb R^n)}
\int_0^1 t^{N_k\lambda}\varphi(t)\,\mathrm{d}t .
\]
Hence, $\mathcal{H}_{\varphi}$ is bounded on
$\dot B_k^{\vec p, \vec q,\lambda}(\mathbb R^n)$ and
\begin{equation*}
\left\|\mathcal{H}_{\varphi}\right\|_{\dot B_k^{\vec p, \vec q,\lambda}(\mathbb R^n)
\to \dot B_k^{\vec p, \vec q,\lambda}(\mathbb R^n)}
\leq
\int_0^1 t^{N_k\lambda}\varphi(t)\,\mathrm{d}t .
\end{equation*}

Conversely, suppose that $\mathcal{H}_{\varphi}$ is bounded on
$\dot B_k^{\vec p, \vec q,\lambda}(\mathbb R^n)$. Since
$\lambda>-\frac1{N_k}\sum\limits_{i=1}^{m}\frac{N_{k^{(i)}}}{p_i},$
we can choose real numbers $\lambda_1,\ldots,\lambda_m$ such that
$\lambda_i>-\frac{N_{k^{(i)}}}{p_i},\; i=1,\ldots,m,$
and
$\sum\limits_{i=1}^{m}\lambda_i=N_k\lambda.$

Let us construct a purely radial function
\[
f_0(x)=f_0(\rho_1\omega_1, \ldots, \rho_m\omega_m)=\prod_{i=1}^{m}\rho_i^{\lambda_i}.
\]

By the homogeneity of $f_0$,
\[
f_0(Rx)=\prod_{i=1}^{m}(R \rho_i)^{\lambda_i}=R^{\sum\limits_{i=1}^{m}\lambda_i}f_0(x)=R^{N_k\lambda}f_0(x),
\]
and hence
\[
\left\|f_0\chi_{B(0,R)}\right\|_{L_{\mathrm{rad}, k}^{\vec p} L_{\mathrm{ang}, k}^{\vec q}(\mathbb R^n)}
=R^{N_k\lambda+\sum\limits_{i=1}^{m}\frac{N_{k^{(i)}}}{p_i}}
\left\|f_0\chi_{B(0,1)}\right\|_{L_{\mathrm{rad}, k}^{\vec p} L_{\mathrm{ang}, k}^{\vec q}(\mathbb R^n)}.
\]
Moreover,
\[
|B(0,R)|_k=\nu_{n,k}R^{N_k}
\]
and
\[
\left\|\chi_{B(0,R)}\right\|_{L_{\mathrm{rad}, k}^{\vec p} L_{\mathrm{ang}, k}^{\vec q}(\mathbb R^n)}
=
R^{\sum\limits_{i=1}^{m}\frac{N_{k^{(i)}}}{p_i}}
\left\|\chi_{B(0,1)}\right\|_{L_{\mathrm{rad}, k}^{\vec p} L_{\mathrm{ang}, k}^{\vec q}(\mathbb R^n)}.
\]
Therefore,
\begin{align*}
\frac{
\left\|f_0\chi_{B(0,R)}\right\|_{L_{\mathrm{rad}, k}^{\vec p} L_{\mathrm{ang}, k}^{\vec q}(\mathbb R^n)}
}{|B(0,R)|_k^\lambda\left\|\chi_{B(0,R)}\right\|_{L_{\mathrm{rad}, k}^{\vec p} L_{\mathrm{ang}, k}^{\vec q}(\mathbb R^n)}}
&=\frac{R^{N_k\lambda+\sum\limits_{i=1}^{m}\frac{N_{k^{(i)}}}{p_i}}\left\|f_0\chi_{B(0,1)}\right\|_{L_{\mathrm{rad}, k}^{\vec p} L_{\mathrm{ang}, k}^{\vec q}(\mathbb R^n)}
}{\nu_{n,k}^\lambda R^{N_k\lambda}R^{\sum\limits_{i=1}^{m}\frac{N_{k^{(i)}}}{p_i}}
\left\|\chi_{B(0,1)}\right\|_{L_{\mathrm{rad}, k}^{\vec p} L_{\mathrm{ang}, k}^{\vec q}(\mathbb R^n)}} \\
&=\frac{\left\|f_0\chi_{B(0,1)}\right\|_{L_{\mathrm{rad}, k}^{\vec p} L_{\mathrm{ang}, k}^{\vec q}(\mathbb R^n)}
}{\nu_{n,k}^\lambda\left\|\chi_{B(0,1)}\right\|_{L_{\mathrm{rad}, k}^{\vec p} L_{\mathrm{ang}, k}^{\vec q}(\mathbb R^n)}}.
\end{align*}

Since $\lambda_i>-\frac{N_{k^{(i)}}}{p_i}$ for every $i$, we have $\lambda_i p_i + N_{k^{(i)}} - 1 > -1$. Thus
\[
\left\|f_0\chi_{B(0,1)}\right\|_{L_{\mathrm{rad}, k}^{\vec p} L_{\mathrm{ang}, k}^{\vec q}(\mathbb R^n)}<\infty.
\]
Consequently,
\[
\left\|f_0\right\|_{\dot B_k^{\vec p, \vec q,\lambda}(\mathbb R^n)}
=
\frac{
\left\|f_0\chi_{B(0,1)}\right\|_{L_{\mathrm{rad}, k}^{\vec p} L_{\mathrm{ang}, k}^{\vec q}(\mathbb R^n)}
}{
\nu_{n,k}^\lambda
\left\|\chi_{B(0,1)}\right\|_{L_{\mathrm{rad}, k}^{\vec p} L_{\mathrm{ang}, k}^{\vec q}(\mathbb R^n)}
}
<\infty.
\]
Hence,
\[
f_0\in\dot B_k^{\vec p, \vec q,\lambda}(\mathbb R^n).
\]

Furthermore, for $x\in\mathbb R^n$, we have
\begin{align*}
\mathcal{H}_{\varphi} f_0(x)
&=\int_0^1 f_0(tx)\varphi(t)\,\mathrm{d}t \\
&=\int_0^1\prod_{i=1}^{m}(t \rho_i)^{\lambda_i}\varphi(t)\,\mathrm{d}t \\
&=\prod_{i=1}^{m}\rho_i^{\lambda_i}\int_0^1 t^{\sum\limits_{i=1}^{m}\lambda_i}\varphi(t)\,\mathrm{d}t \\
&=f_0(x)\int_0^1 t^{N_k\lambda}\varphi(t)\,\mathrm{d}t.
\end{align*}
Since $\mathcal{H}_{\varphi}$ is bounded on
$\dot B_k^{\vec p, \vec q,\lambda}(\mathbb R^n)$ and
$f_0\not\equiv0$, it follows that
\[
\int_0^1 t^{N_k\lambda}\varphi(t)\,\mathrm{d}t<\infty.
\]
Moreover,
\begin{align*}
\left\|\mathcal{H}_{\varphi}\right\|_{\dot B_k^{\vec p, \vec q,\lambda}(\mathbb R^n)\to \dot B_k^{\vec p, \vec q,\lambda}(\mathbb R^n)}
\geq\frac{\left\|\mathcal{H}_{\varphi} f_0\right\|_{\dot B_k^{\vec p, \vec q,\lambda}(\mathbb R^n)}}{\left\|f_0\right\|_{\dot B_k^{\vec p, \vec q,\lambda}(\mathbb R^n)}}
=\int_0^1 t^{N_k\lambda}\varphi(t)\,\mathrm{d}t .
\end{align*}
The proof of Theorem \ref{2} is completed.
\end{proof}

\section{Boundednesss on $\mathcal{H}_{\varphi}^{b}$ on $\dot{M}_{\vec r, \vec{\tilde r}, k}^p(\mathbb{R}^n)$ and $\dot{B}^{\vec r, \vec{\tilde r},\lambda}_k(\mathbb R^n)$}\label{sec4}

Now we study the condition on $\varphi$ such that $\mathcal{H}_{\varphi}^{b}$ is bounded  on $\dot{M}_{\vec r, \vec{\tilde r}, k}^p(\mathbb{R}^n)$ and $\dot{B}^{\vec r, \vec{\tilde r},\lambda}_k(\mathbb R^n)$. Our result can be read as follows.

\begin{theorem}\label{3}
Let $0<p<\infty$. Let $\vec s=(s_1,\ldots,s_m)$, $\vec r=(r_1,\ldots,r_m)$, and $\vec q=(q_1,\ldots,q_m)$ be radial indices, and let $\vec{\tilde s}=(\tilde s_1,\ldots,\tilde s_m)$, $\vec{\tilde r}=(\tilde r_1,\ldots,\tilde r_m)$, and $\vec{\tilde q}=(\tilde q_1,\ldots,\tilde q_m)$ be angular indices satisfying
$1<\vec s<\vec r<\infty,\:1<\vec{\tilde s}<\vec{\tilde r}<\infty,\:1<\vec q, \vec{\tilde q}<\infty,$
and $\frac1{s_i}=\frac1{r_i}+\frac1{q_i},\:\frac1{\tilde{s}_i}=\frac1{\tilde{r}_i}+\frac1{\tilde{q}_i},\: i=1,\ldots,m.$
Let $\sum\limits_{i=1}^m\frac{N_{k^{(i)}}}{r_i}> \frac{N_k}{p}$ and $\varphi$ be a non-negative integrable function on $(0,1]$. Then $\mathcal{H}_{\varphi}^{b}$ is bounded from
$\dot{M}_{\vec r, \vec{\tilde r}, k}^p(\mathbb{R}^n)$ to
$\dot{M}_{\vec s, \vec{\tilde s}, k}^p(\mathbb{R}^n)$ for all
$b\in \mathrm{CMO}_{\vec q, \vec{\tilde q}, k}(\mathbb{R}^n)$ if and only if
\begin{equation}
\int_0^1t^{-\frac{N_k}{p}}\varphi(t)\log\frac{2}{t}\,\mathrm{d}t<\infty.\label{4.1}
\end{equation}
\end{theorem}

\begin{theorem}\label{4}
Suppose $1<\vec r<\vec r_1<\infty$ and $1<\vec{\tilde r}<\vec{\tilde r}_1<\infty$, and let
$\frac1{r_i}=\frac1{r_{1i}}+\frac1{r_{2i}}, \,\frac1{\tilde{r}_i}=\frac1{\tilde{r}_{1i}}+\frac1{\tilde{r}_{2i}},\, i=1,\ldots,m.$
Let $-\frac1{N_k}\sum\limits_{i=1}^{m}\frac{N_{k^{(i)}}}{r_{1i}}<\lambda.$
Further suppose that $\varphi$ is a non-negative locally integrable function on $(0,1]$.
Then $\mathcal{H}_{\varphi}^{b}$ is bounded from
$\dot{B}^{\vec r_1, \vec{\tilde r}_1,\lambda}_k(\mathbb R^n)$ to
$\dot{B}^{\vec r, \vec{\tilde r},\lambda}_k(\mathbb R^n)$ for all
$b\in \mathrm{CMO}_{\vec r_2, \vec{\tilde r}_2, k}(\mathbb R^n)$ if and only if
\begin{equation}
\int_0^1 t^{N_k\lambda}\log\frac1t\,\varphi(t)\,\mathrm{d}t<\infty.
\label{12}
\end{equation}
\end{theorem}

\begin{theorem}\label{5}
Suppose $1<\vec r<\vec r_1<\infty$ and $1<\vec{\tilde r}<\vec{\tilde r}_1<\infty$, and let
$\frac1{r_i}=\frac1{r_{1i}}+\frac1{r_{2i}}, \, \frac1{\tilde{r}_i}=\frac1{\tilde{r}_{1i}}+\frac1{\tilde{r}_{2i}},\, i=1,\ldots,m.$
Let $-\frac1{N_k}\sum\limits_{i=1}^{m}\frac{N_{k^{(i)}}}{r_i}<\lambda,\,-\frac1{N_k}\sum\limits_{i=1}^{m}\frac{N_{k^{(i)}}}{r_{1i}}<\lambda_1,\,0<\lambda_2<\frac1{N_k},$
and suppose that $\lambda=\lambda_1+\lambda_2.$
Further suppose that $\varphi$ is a non-negative function on $(0,1]$.
Then $\mathcal{H}_{\varphi}^{b}$ is bounded from
$\dot{B}^{\vec r_1, \vec{\tilde r}_1,\lambda_1}_k(\mathbb R^n)$ to
$\dot{B}^{\vec r, \vec{\tilde r},\lambda}_k(\mathbb R^n)$ for all
$b\in \mathrm{CMO}_{\vec r_2, \vec{\tilde r}_2,\lambda_2,k}(\mathbb R^n)$ if
\begin{equation}
\int_0^1 t^{N_k\lambda_1}\varphi(t)\,\mathrm{d}t<\infty.
\label{13}
\end{equation}
\end{theorem}

\begin{proof}[Proof of Theorem 4.1]
Suppose that (\ref{4.1}) holds. We shall prove that $\mathcal{H}_{\varphi}^{b}$ is bounded. For any ball $B(0,R)$, $R>0$, we get
\begin{align*}
&\|\mathcal{H}_{\varphi}^{b}f\|_{\dot{M}_{\vec s, \vec{\tilde s}, k}^p(\mathbb{R}^n)}\\
&=\sup_{R>0}|B(0,R)|_k^{\frac1p-\frac1{N_k}\left(\sum\limits_{i=1}^m\frac{N_{k^{(i)}}}{s_i}\right)}\|\mathcal{H}_{\varphi}^{b}f\cdot\chi_{B(0,R)}\|_{L_{\mathrm{rad}, k}^{\vec s} L_{\mathrm{ang}, k}^{\vec{\tilde s}}(\mathbb{R}^n)}\\
&\leqslant\sup_{R>0}|B(0,R)|_k^{\frac1p-\frac1{N_k}\left(\sum\limits_{i=1}^m\frac{N_{k^{(i)}}}{s_i}\right)}\left\|\int_0^1(b(\cdot)-b_{B(0,R),k})f(t\cdot)\varphi(t)\,\mathrm{d}t\cdot\chi_{B(0,R)}(\cdot)\right\|_{L_{\mathrm{rad}, k}^{\vec s} L_{\mathrm{ang}, k}^{\vec{\tilde s}}(\mathbb{R}^n)}\\
&\quad+\sup_{R>0}|B(0,R)|_k^{\frac1p-\frac1{N_k}\left(\sum\limits_{i=1}^m\frac{N_{k^{(i)}}}{s_i}\right)}\left\|\int_0^1(b_{B(0,R),k}-b_{B(0,tR),k})f(t\cdot)\varphi(t)\,\mathrm{d}t\cdot\chi_{B(0,R)}(\cdot)
\right\|_{L_{\mathrm{rad}, k}^{\vec s} L_{\mathrm{ang}, k}^{\vec{\tilde s}}(\mathbb{R}^n)}\\
&\quad+\sup_{R>0}|B(0,R)|_k^{\frac1p-\frac1{N_k}\left(\sum\limits_{i=1}^m\frac{N_{k^{(i)}}}{s_i}\right)}
\left\|\int_0^1(b_{B(0,tR),k}-b(t\cdot))f(t\cdot)\varphi(t)\,\mathrm{d}t\cdot\chi_{B(0,R)}(\cdot)\right\|_{L_{\mathrm{rad}, k}^{\vec s} L_{\mathrm{ang}, k}^{\vec{\tilde s}}(\mathbb{R}^n)}\\
&:=I+II+III,
\end{align*}
here and below
\[
b_{B(0,R),k}
=
\frac{1}{|B(0,R)|_k}
\int_{B(0,R)}b(x)\,\mathrm{d}\mu_k(x).
\]

For the first term $I$, by H\"older's inequality and the scaling equivalence
\[
\|\chi_{B(0,R)}\|_{L_{\mathrm{rad}, k}^{\vec q} L_{\mathrm{ang}, k}^{\vec{\tilde q}}(\mathbb{R}^n)}
\sim
|B(0,R)|_k^{\frac1{N_k}
\left(\sum_{i=1}^m\frac{N_{k^{(i)}}}{q_i}\right)},
\] we derive
\begin{align*}
I&=\sup_{R>0}|B(0,R)|_k^{\frac1p-\frac1{N_k}\left(\sum\limits_{i=1}^m\frac{N_{k^{(i)}}}{s_i}\right)}
\left\|\mathcal{H}_{\varphi}f(\cdot)(b(\cdot)-b_{B(0,R),k})\chi_{B(0,R)}(\cdot)\right\|_{L_{\mathrm{rad}, k}^{\vec s} L_{\mathrm{ang}, k}^{\vec{\tilde s}}(\mathbb{R}^n)}\\
&\leqslant\sup_{R>0}|B(0,R)|_k^{\frac1p-\frac1{N_k}\left(\sum\limits_{i=1}^m\frac{N_{k^{(i)}}}{s_i}\right)}\\
&\quad\times\left\|\mathcal{H}_{\varphi}f(\cdot)\chi_{B(0,R)}(\cdot)\right\|_{L_{\mathrm{rad}, k}^{\vec r} L_{\mathrm{ang}, k}^{\vec{\tilde r}}(\mathbb{R}^n)}
\left\|(b(\cdot)-b_{B(0,R),k})\chi_{B(0,R)}(\cdot)\right\|_{L_{\mathrm{rad}, k}^{\vec q} L_{\mathrm{ang}, k}^{\vec{\tilde q}}(\mathbb{R}^n)}\\
&\leqslant\sup_{R>0}|B(0,R)|_k^{\frac1p-\frac1{N_k}\left(\sum\limits_{i=1}^m\frac{N_{k^{(i)}}}{r_i}\right)}
\left\|\mathcal{H}_{\varphi}f(\cdot)\chi_{B(0,R)}(\cdot)\right\|_{L_{\mathrm{rad}, k}^{\vec r} L_{\mathrm{ang}, k}^{\vec{\tilde r}}(\mathbb{R}^n)}\\
&\quad\times|B(0,R)|_k^{-\frac1{N_k}\left(\sum\limits_{i=1}^m\frac{N_{k^{(i)}}}{q_i}\right)}
\left\|(b(\cdot)-b_{B(0,R),k})\chi_{B(0,R)}(\cdot)\right\|_{L_{\mathrm{rad}, k}^{\vec q} L_{\mathrm{ang}, k}^{\vec{\tilde q}}(\mathbb{R}^n)}\\
&\leqslant\int_0^1 t^{-\frac{N_k}{p}}\varphi(t)\,\mathrm{d}t\,\|b\|_{\mathrm{CMO}_{\vec q, \vec{\tilde q}, k}(\mathbb{R}^n)}\|f\|_{\dot{M}_{\vec r, \vec{\tilde r}, k}^p(\mathbb{R}^n)}.
\end{align*}

Moreover,
\[
\int_0^1 t^{-\frac{N_k}{p}}\varphi(t)\,\mathrm{d}t
\lesssim\int_0^1t^{-\frac{N_k}{p}}\varphi(t)\log\frac{2}{t}\,\mathrm{d}t.
\]

For the last term $III$, by using Minkowski's inequality and H\"older's inequality, we obtain
\begin{align*}
III
&=\sup_{R>0}|B(0,R)|_k^{\frac1p-\frac1{N_k}\left(\sum_{i=1}^m\frac{N_{k^{(i)}}}{s_i}\right)}\\
&\quad\times\left\|\int_0^1(b_{B(0,tR),k}-b(t\cdot))f(t\cdot)\varphi(t)\,\mathrm{d}t
\cdot\chi_{B(0,R)}(\cdot)\right\|_{L_{\mathrm{rad}, k}^{\vec s} L_{\mathrm{ang}, k}^{\vec{\tilde s}}(\mathbb{R}^n)}\\
&\leqslant\sup_{R>0}|B(0,R)|_k^{\frac1p-\frac1{N_k}\left(\sum_{i=1}^m\frac{N_{k^{(i)}}}{s_i}\right)}\\
&\quad\times\int_0^1\varphi(t)\left\|(b_{B(0,tR),k}-b(t\cdot))f(t\cdot)\chi_{B(0,R)}(\cdot)
\right\|_{L_{\mathrm{rad}, k}^{\vec s} L_{\mathrm{ang}, k}^{\vec{\tilde s}}(\mathbb{R}^n)}\,\mathrm{d}t\\
&\leqslant\sup_{R>0}|B(0,R)|_k^{\frac1p-\frac1{N_k}\left(\sum_{i=1}^m\frac{N_{k^{(i)}}}{s_i}\right)}\\
&\quad\times\int_0^1\varphi(t)\left\|f(t\cdot)\chi_{B(0,R)}(\cdot)
\right\|_{L_{\mathrm{rad}, k}^{\vec r} L_{\mathrm{ang}, k}^{\vec{\tilde r}}(\mathbb{R}^n)}
\left\|(b_{B(0,tR),k}-b(t\cdot))\chi_{B(0,R)}(\cdot)
\right\|_{L_{\mathrm{rad}, k}^{\vec q} L_{\mathrm{ang}, k}^{\vec{\tilde q}}(\mathbb{R}^n)}\,\mathrm{d}t\\
&=\sup_{R>0}\int_0^1t^{-\frac{N_k}{p}}\varphi(t)|B(0,tR)|_k^{\frac1p-\frac1{N_k}\left(\sum_{i=1}^m\frac{N_{k^{(i)}}}{s_i}\right)}\\
&\quad\times\left\|f(\cdot)\chi_{B(0,tR)}(\cdot)\right\|_{L_{\mathrm{rad}, k}^{\vec r} L_{\mathrm{ang}, k}^{\vec{\tilde r}}(\mathbb{R}^n)}
\left\|(b_{B(0,tR),k}-b(\cdot))\chi_{B(0,tR)}(\cdot)\right\|_{L_{\mathrm{rad}, k}^{\vec q} L_{\mathrm{ang}, k}^{\vec{\tilde q}}(\mathbb{R}^n)}\,\mathrm{d}t\\
&\leqslant\int_0^1t^{-\frac{N_k}{p}}\varphi(t)\,\mathrm{d}t\,\|b\|_{\mathrm{CMO}_{\vec q, \vec{\tilde q}, k}(\mathbb{R}^n)}\|f\|_{\dot{M}_{\vec r, \vec{\tilde r}, k}^p(\mathbb{R}^n)}.
\end{align*}

For the term $II$, using Minkowski's inequality and H\"older's inequality again, one has
\begin{align*}
II &= \sup_{R>0}|B(0,R)|_k^{\frac1p-\frac1{N_k}\left(\sum\limits_{i=1}^m\frac{N_{k^{(i)}}}{s_i}\right)} \\
&\quad\times\left\|\int_0^1(b_{B(0,R),k}-b_{B(0,tR),k}) f(t\cdot)\varphi(t)\,\mathrm{d}t\cdot\chi_{B(0,R)}(\cdot)\right\|_{L_{\mathrm{rad}, k}^{\vec s} L_{\mathrm{ang}, k}^{\vec{\tilde s}}(\mathbb{R}^n)}\\
&\leqslant\sup_{R>0}|B(0,R)|_k^{\frac1p-\frac1{N_k}\left(\sum\limits_{i=1}^m\frac{N_{k^{(i)}}}{s_i}\right)}\\
&\quad\times\int_0^1\varphi(t)|b_{B(0,R),k}-b_{B(0,tR),k}|\left\|f(t\cdot)\chi_{B(0,R)}(\cdot)\right\|_{L_{\mathrm{rad}, k}^{\vec s} L_{\mathrm{ang}, k}^{\vec{\tilde s}}(\mathbb{R}^n)}\,\mathrm{d}t\\
&\leqslant\sup_{R>0}|B(0,R)|_k^{\frac1p-\frac1{N_k}\left(\sum\limits_{i=1}^m\frac{N_{k^{(i)}}}{s_i}\right)}\\
&\quad\times\int_0^1\varphi(t)|b_{B(0,R),k}-b_{B(0,tR),k}| \left\|f(t\cdot)\chi_{B(0,R)}(\cdot)\right\|_{L_{\mathrm{rad}, k}^{\vec r} L_{\mathrm{ang}, k}^{\vec{\tilde r}}(\mathbb{R}^n)} \left\|\chi_{B(0,R)}(\cdot)\right\|_{L_{\mathrm{rad}, k}^{\vec q} L_{\mathrm{ang}, k}^{\vec{\tilde q}}(\mathbb{R}^n)} \,\mathrm{d}t\\
&\leqslant\sup_{R>0}|B(0,R)|_k^{\frac1p-\frac1{N_k}\left(\sum\limits_{i=1}^m\frac{N_{k^{(i)}}}{r_i}\right)}\\
&\quad\times\int_0^1\varphi(t)|b_{B(0,R),k}-b_{B(0,tR),k}|\left\|f(t\cdot)\chi_{B(0,R)}(\cdot)\right\|_{L_{\mathrm{rad}, k}^{\vec r} L_{\mathrm{ang}, k}^{\vec{\tilde r}}(\mathbb{R}^n)}\,\mathrm{d}t\\
&= \sup_{R>0}|B(0,R)|_k^{\frac1p-\frac1{N_k}\left(\sum\limits_{i=1}^m\frac{N_{k^{(i)}}}{r_i}\right)}\\
&\quad\times\int_0^1\varphi(t)|b_{B(0,R),k}-b_{B(0,tR),k}| \, t^{-\sum\limits_{i=1}^m\frac{N_{k^{(i)}}}{r_i}} \left\|f(\cdot)\chi_{B(0,tR)}(\cdot)\right\|_{L_{\mathrm{rad}, k}^{\vec r} L_{\mathrm{ang}, k}^{\vec{\tilde r}}(\mathbb{R}^n)}\,\mathrm{d}t\\
&\leqslant \sup_{R>0}|B(0,R)|_k^{\frac1p-\frac1{N_k}\left(\sum\limits_{i=1}^m\frac{N_{k^{(i)}}}{r_i}\right)}\\
&\quad\times\int_0^1\varphi(t)|b_{B(0,R),k}-b_{B(0,tR),k}| \, t^{-\sum\limits_{i=1}^m\frac{N_{k^{(i)}}}{r_i}} \|f\|_{\dot{M}_{\vec r, \vec{\tilde r}, k}^p(\mathbb{R}^n)} |B(0,tR)|_k^{-\frac1p+\frac1{N_k}\left(\sum\limits_{i=1}^m\frac{N_{k^{(i)}}}{r_i}\right)}\,\mathrm{d}t\\
&= \|f\|_{\dot{M}_{\vec r, \vec{\tilde r}, k}^p(\mathbb{R}^n)} \sup_{R>0} \int_0^1\varphi(t)|b_{B(0,R),k}-b_{B(0,tR),k}| \, t^{-\sum\limits_{i=1}^m\frac{N_{k^{(i)}}}{r_i}} \left( t^{N_k} \right)^{-\frac1p+\frac1{N_k}\left(\sum\limits_{i=1}^m\frac{N_{k^{(i)}}}{r_i}\right)} \,\mathrm{d}t\\
&= \|f\|_{\dot{M}_{\vec r, \vec{\tilde r}, k}^p(\mathbb{R}^n)} \sup_{R>0}\int_0^1 t^{-\frac{N_k}{p}} \varphi(t)|b_{B(0,R),k}-b_{B(0,tR),k}|\,\mathrm{d}t.
\end{align*}
For any $0<t<1$, there exists some $l\in\mathbb{N}\setminus\{0\}$ such that
$
2^{-l}<t\leqslant 2^{-l+1}.
$
Then
\begin{align*}
II&\leqslant\|f\|_{\dot{M}_{\vec r, \vec{\tilde r}, k}^p(\mathbb{R}^n)}\sup_{R>0}\int_0^1
t^{-\frac{N_k}{p}}\varphi(t)\\
&\quad\times\left(\sum_{v=1}^{l}|b_{B(0,2^{-v}R),k}-b_{B(0,2^{-v+1}R),k}|+|b_{B(0,2^{-l}R),k}-b_{B(0,tR),k}|\right)\,\mathrm{d}t.
\end{align*}

We note that
\begin{align*}
&\left\| \chi_{B(0,2^{-v+1}R)}(\cdot) \right\|_{L_{\mathrm{rad}, k}^{\vec q} L_{\mathrm{ang}, k}^{\vec{\tilde q}}(\mathbb{R}^n)}\left\| \chi_{B(0,2^{-v+1}R)}(\cdot) \right\|_{L_{\mathrm{rad}, k}^{\vec q'} L_{\mathrm{ang}, k}^{\vec{\tilde q}'}(\mathbb{R}^n)}\\
&\lesssim\vert{}B(0,R)\vert{}_k^{\frac{1}{N_k}\sum\limits_{i=1}^m \frac{N_{k^{(i)}}}{q_i}} \vert{}B(0,R)\vert{}_k^{\frac{1}{N_k}\sum\limits_{i=1}^m \frac{N_{k^{(i)}}}{q'_i}}\\
&=  \vert{}B(0,R)\vert{}_k.
\end{align*}
Let $\frac{1}{\vec q}+\frac{1}{(\vec q)'}=1$ and $\frac{1}{\vec{\tilde q}}+\frac{1}{(\vec{\tilde q})'}=1$, for any $v\in\mathbb{N}$, we have
\begin{align*}
&|b_{B(0,2^{-v}R),k}-b_{B(0,2^{-v+1}R),k}|\\
&=\frac{1}{|B(0,2^{-v}R)|_k}\left|\int_{B(0,2^{-v}R)}\left(b(x)-b_{B(0,2^{-v+1}R),k}\right)\,\mathrm{d}\mu_k(x)\right|\\
&\lesssim\frac{1}{|B(0,2^{-v+1}R)|_k}\int_{B(0,2^{-v+1}R)}|b(x)-b_{B(0,2^{-v+1}R),k}|\,\mathrm{d}\mu_k(x)\\
&\leqslant\frac{1}{|B(0,2^{-v+1}R)|_k}\left\|\left(b(\cdot)-b_{B(0,2^{-v+1}R),k}\right)\chi_{B(0,2^{-v+1}R)}(\cdot)\right\|_{L_{\mathrm{rad}, k}^{\vec q} L_{\mathrm{ang}, k}^{\vec{\tilde q}}(\mathbb{R}^n)}\\
&\quad\times\left\| \chi_{B(0,2^{-v+1}R)}(\cdot) \right\|_{L_{\mathrm{rad}, k}^{(\vec q)'} L_{\mathrm{ang}, k}^{(\vec{\tilde q})'}(\mathbb{R}^n)}\\
&\lesssim\frac{1}{\left\| \chi_{B(0,2^{-v+1}R)}(\cdot) \right\|_{L_{\mathrm{rad}, k}^{\vec q} L_{\mathrm{ang}, k}^{\vec{\tilde q}}(\mathbb{R}^n)}}\\
&\quad\times\left\|\left(b(\cdot)-b_{B(0,2^{-v+1}R),k}\right)\chi_{B(0,2^{-v+1}R)}(\cdot)\right\|_{L_{\mathrm{rad}, k}^{\vec q} L_{\mathrm{ang}, k}^{\vec{\tilde q}}(\mathbb{R}^n)}\\
&\leqslant\|b\|_{\mathrm{CMO}_{\vec q, \vec{\tilde q}, k}(\mathbb{R}^n)}.
\end{align*}
Similarly, since $2^{-l}<t\leqslant2^{-l+1}$, we also have
\begin{align*}
&|b_{B(0,2^{-l}R),k}-b_{B(0,tR),k}|\\
&= \frac{1}{|B(0,2^{-l}R)|_k}\left|\int_{B(0,2^{-l}R)}\left(b(x)-b_{B(0,tR),k}\right)\,\mathrm{d}\mu_k(x)\right|\\
&\lesssim \frac{1}{|B(0,tR)|_k}\int_{B(0,tR)}|b(x)-b_{B(0,tR),k}|\,\mathrm{d}\mu_k(x)\\
&\leqslant \frac{1}{|B(0,tR)|_k}\left\|\left(b(\cdot)-b_{B(0,tR),k}\right)\chi_{B(0,tR)}(\cdot)\right\|_{L_{\mathrm{rad}, k}^{\vec q} L_{\mathrm{ang}, k}^{\vec{\tilde q}}(\mathbb{R}^n)}\left\| \chi_{B(0,tR)}(\cdot) \right\|_{L_{\mathrm{rad}, k}^{(\vec q)'} L_{\mathrm{ang}, k}^{(\vec{\tilde q})'}(\mathbb{R}^n)}\\
&\lesssim \frac{1}{\left\| \chi_{B(0,tR)}(\cdot) \right\|_{L_{\mathrm{rad}, k}^{\vec q} L_{\mathrm{ang}, k}^{\vec{\tilde q}}(\mathbb{R}^n)}}\left\|\left(b(\cdot)-b_{B(0,tR),k}\right)\chi_{B(0,tR)}(\cdot)\right\|_{L_{\mathrm{rad}, k}^{\vec q} L_{\mathrm{ang}, k}^{\vec{\tilde q}}(\mathbb{R}^n)}\\
&\leqslant \|b\|_{\mathrm{CMO}_{\vec q, \vec{\tilde q}, k}(\mathbb{R}^n)}.
\end{align*}
Therefore,
\begin{align*}
II&\lesssim\|b\|_{\mathrm{CMO}_{\vec q, \vec{\tilde q}, k}(\mathbb{R}^n)}
\|f\|_{\dot{M}_{\vec r, \vec{\tilde r}, k}^p(\mathbb{R}^n)}\sum_{l=1}^{\infty}\int_{2^{-l}}^{2^{-l+1}}t^{-\frac{N_k}{p}}\varphi(t)(l+1)\,\mathrm{d}t\\
&\lesssim\|b\|_{\mathrm{CMO}_{\vec q, \vec{\tilde q}, k}(\mathbb{R}^n)}\|f\|_{\dot{M}_{\vec r, \vec{\tilde r}, k}^p(\mathbb{R}^n)}\sum_{l=1}^{\infty}
\int_{2^{-l}}^{2^{-l+1}}t^{-\frac{N_k}{p}}\varphi(t)(\log 2^l+1)\,\mathrm{d}t\\
&\lesssim\|b\|_{\mathrm{CMO}_{\vec q, \vec{\tilde q}, k}(\mathbb{R}^n)}
\|f\|_{\dot{M}_{\vec r, \vec{\tilde r}, k}^p(\mathbb{R}^n)}\sum_{l=1}^{\infty}\int_{2^{-l}}^{2^{-l+1}}t^{-\frac{N_k}{p}}\varphi(t)
\left(\log\frac1t+1\right)\,\mathrm{d}t\\
&\lesssim\|b\|_{\mathrm{CMO}_{\vec q, \vec{\tilde q}, k}(\mathbb{R}^n)}\|f\|_{\dot{M}_{\vec r, \vec{\tilde r}, k}^p(\mathbb{R}^n)}
\int_0^1t^{-\frac{N_k}{p}}\varphi(t)\left(\log\frac1t+1\right)\,\mathrm{d}t\\
&\lesssim\|b\|_{\mathrm{CMO}_{\vec q, \vec{\tilde q}, k}(\mathbb{R}^n)}\|f\|_{\dot{M}_{\vec r, \vec{\tilde r}, k}^p(\mathbb{R}^n)}\int_0^1t^{-\frac{N_k}{p}}\varphi(t)
\log\frac{2}{t}\,\mathrm{d}t.
\end{align*}
Combining all the estimates of $I$, $II$ and $III$, it yields
\[
\|\mathcal{H}_{\varphi}^{b}f\|_{\dot{M}_{\vec s, \vec{\tilde s}, k}^p(\mathbb{R}^n)}
\lesssim
\int_0^1t^{-\frac{N_k}{p}}\varphi(t)\log\frac{2}{t}\,\mathrm{d}t\|b\|_{\mathrm{CMO}_{\vec q, \vec{\tilde q}, k}(\mathbb{R}^n)}
\|f\|_{\dot{M}_{\vec r, \vec{\tilde r}, k}^p(\mathbb{R}^n)}.
\]

To prove the converse, suppose that $\mathcal{H}_{\varphi}^{b}$ is bounded from $\dot{M}_{\vec r, \vec{\tilde r}, k}^p(\mathbb{R}^n)$ to $\dot{M}_{\vec s, \vec{\tilde s}, k}^p(\mathbb{R}^n)$.
We take a radial function
\[
f_0(x) = f_0(\rho_1\omega_1, \ldots, \rho_m\omega_m)
=\prod_{i=1}^m\rho_i^{-\frac{N_{k^{(i)}}}{\alpha_i}},
\]
where $r_i<\alpha_i<\infty,\,i=1,\ldots,m$, and $\sum_{i=1}^m\frac{N_{k^{(i)}}}{\alpha_i}=\frac{N_k}{p}$.
Such $\alpha_i$ exist since $\sum\limits_{i=1}^m\frac{N_{k^{(i)}}}{r_i}>\frac{N_k}{p}$.

As in the proof of Theorem \ref{th1}, we know that
$f_0\in \dot{M}_{\vec r, \vec{\tilde r}, k}^p(\mathbb{R}^n).$
It is not hard to verify that
$f_0\in \dot{M}_{\vec s, \vec{\tilde s}, k}^p(\mathbb{R}^n),$
since $1<\vec s<\vec r<\vec \alpha<\infty$ and $1<\vec{\tilde s}<\vec{\tilde r}<\infty$.

Take
\[
b_0(x)=\log |x|,
\qquad x\in\mathbb{R}^n,
\]
and define $b_0(0)=0$. To show that $b_0 \in \mathrm{CMO}_{\vec q, \vec{\tilde q}, k}(\mathbb{R}^n)$, we must demonstrate that the supremum over $R>0$ of its normalized central oscillation is finite.
Through simple calculation, we get
\begin{align*}
\int_{B(0, R)} \log|x| \,\mathrm{d}\mu_k(x)
&= \omega_{n,k} \int_0^R \log\rho \cdot \rho^{N_k-1} \,\mathrm{d}\rho \\
&= \omega_{n,k} \left[ \frac{\rho^{N_k}}{N_k}\log\rho \right]_0^R - \omega_{n,k} \int_0^R \frac{\rho^{N_k-1}}{N_k} \,\mathrm{d}\rho \\
&= \frac{\omega_{n,k} R^{N_k}}{N_k} \log R - \frac{\omega_{n,k} R^{N_k}}{N_k^2}.
\end{align*}
Since $|B(0, R)|_k = \frac{\omega_{n,k}}{N_k} R^{N_k}$,
\[
(b_0)_{B(0, R), k} = \frac{1}{|B(0, R)|_k} \int_{B(0, R)} \log|x| \,\mathrm{d}\mu_k(x) = \log R - \frac{1}{N_k}.
\]

Consequently,
\[
b_0(x) - (b_0)_{B(0, R), k} = \log|x| - \log R + \frac{1}{N_k} = \log\frac{|x|}{R} + \frac{1}{N_k}.
\]

Utilizing $x = Ry$, we obtain that
\begin{align*}
\left\| \big(b_0(\cdot) - (b_0)_{B(0, R), k}\big) \chi_{B(0, R)}(\cdot) \right\|_{L_{\mathrm{rad}, k}^{\vec q} L_{\mathrm{ang}, k}^{\vec{\tilde q}}(\mathbb{R}^n)}= R^{\sum\limits_{i=1}^m \frac{N_{k^{(i)}}}{q_i}} \left\| \left(\log|y| + \frac{1}{N_k}\right) \chi_{B(0, 1)}(\cdot) \right\|_{L_{\mathrm{rad}, k}^{\vec q} L_{\mathrm{ang}, k}^{\vec{\tilde q}}(\mathbb{R}^n)}.
\end{align*}
We also have
\begin{align*}
\|\chi_{B(0, R)}\|_{L_{\mathrm{rad}, k}^{\vec q} L_{\mathrm{ang}, k}^{\vec{\tilde q}}(\mathbb{R}^n)} = R^{\sum\limits_{i=1}^m \frac{N_{k^{(i)}}}{q_i}} \|\chi_{B(0, 1)}\|_{L_{\mathrm{rad}, k}^{\vec q} L_{\mathrm{ang}, k}^{\vec{\tilde q}}(\mathbb{R}^n)}.
\end{align*}

By Minkowski's inequality, we can get
\begin{align*}
&\frac{\left\| \big(b_0 - (b_0)_{B(0, R), k}\big) \chi_{B(0, R)} \right\|_{L_{\mathrm{rad}, k}^{\vec q} L_{\mathrm{ang}, k}^{\vec{\tilde q}}}}{\|\chi_{B(0, R)}\|_{L_{\mathrm{rad}, k}^{\vec q} L_{\mathrm{ang}, k}^{\vec{\tilde q}}}}\\
&= \frac{\left\| \left(\log|\cdot| + \frac{1}{N_k}\right) \chi_{B(0, 1)} \right\|_{L_{\mathrm{rad}, k}^{\vec q} L_{\mathrm{ang}, k}^{\vec{\tilde q}}}}{\|\chi_{B(0, 1)}\|_{L_{\mathrm{rad}, k}^{\vec q} L_{\mathrm{ang}, k}^{\vec{\tilde q}}}}\\
&\leqslant \frac{\left\| \log|\cdot| \chi_{B(0, 1)} \right\|_{L_{\mathrm{rad}, k}^{\vec q} L_{\mathrm{ang}, k}^{\vec{\tilde q}}}}{\|\chi_{B(0, 1)}\|_{L_{\mathrm{rad}, k}^{\vec q} L_{\mathrm{ang}, k}^{\vec{\tilde q}}}} + \frac{\frac{1}{N_k} \|\chi_{B(0, 1)}\|_{L_{\mathrm{rad}, k}^{\vec q} L_{\mathrm{ang}, k}^{\vec{\tilde q}}}}{\|\chi_{B(0, 1)}\|_{L_{\mathrm{rad}, k}^{\vec q} L_{\mathrm{ang}, k}^{\vec{\tilde q}}}} \\
&= \frac{\left\| \log|\cdot| \chi_{B(0, 1)} \right\|_{L_{\mathrm{rad}, k}^{\vec q} L_{\mathrm{ang}, k}^{\vec{\tilde q}}}}{\|\chi_{B(0, 1)}\|_{L_{\mathrm{rad}, k}^{\vec q} L_{\mathrm{ang}, k}^{\vec{\tilde q}}}} + \frac{1}{N_k}.
\end{align*}

Note that for $x = (x_1, \dots, x_m) \in \mathbb{R}^n$ with $x_i = \rho_i \omega_i$, the global radius is $|x| = \left(\sum\limits_{i=1}^m \rho_i^2\right)^{\frac{1}{2}}$. Since $0 \leqslant \rho_i \leqslant |x| < 1$ within the unit ball $B(0,1)$, we have
\[
|\log|x|| \leqslant |\log \rho_1|.
\]

We obtain that
\begin{align*}
&\left\| \log|\cdot| \chi_{B(0, 1)} \right\|_{L_{\mathrm{rad}, k}^{\vec q} L_{\mathrm{ang}, k}^{\vec{\tilde q}}} \\
&\leqslant \left\| |\log \rho_1| \chi_{B(0, 1)} \right\|_{L_{\mathrm{rad}, k}^{\vec q} L_{\mathrm{ang}, k}^{\vec{\tilde q}}} \\
&= \Biggl( \int_0^\infty \cdots \Biggl( \int_{\mathbb{S}^{n_2-1}} \Biggl( \int_0^\infty \chi_{\{\sum\limits_{i=1}^m \rho_i^2 < 1\}} |\log \rho_1|^{q_1} \\
&\quad\times \Biggl( \int_{\mathbb{S}^{n_1-1}} 1 \,\mathrm{d}\sigma_{k^{(1)}}(\omega_1) \Biggr)^{\frac{q_1}{\tilde{q}_1}} \rho_1^{N_{k^{(1)}}-1} \,\mathrm{d}\rho_1 \Biggr)^{\frac{q_2}{q_1}} \,\mathrm{d}\sigma_{k^{(2)}}(\omega_2) \Biggr)^{\frac{q_2}{\tilde{q}_2}} \cdots \rho_m^{N_{k^{(m)}}-1} \,\mathrm{d}\rho_m  \Biggr)^{\frac{1}{q_m}}\\
&= \left( \prod_{i=1}^m \omega_{n_i, k^{(i)}}^{\frac{1}{\tilde{q}_i}} \right)
\left( \int_0^1 |\log \rho_1|^{q_1} \rho_1^{N_{k^{(1)}}-1} \,\mathrm{d}\rho_1 \right)^{\frac{1}{q_1}}\prod_{j=2}^m \left( \int_0^1 \rho_j^{N_{k^{(j)}}-1} \,\mathrm{d}\rho_j \right)^{\frac{1}{q_j}}.
\end{align*}

Since $N_{k^{(j)}} > 0$, we note that
\begin{align*}
\int_0^1 \rho_j^{N_{k^{(j)}}-1} \,\mathrm{d}\rho_j = \frac{1}{N_{k^{(j)}}} < \infty.
\end{align*}

Let $u = -\log \rho_1$ and $v = N_{k^{(1)}} u$, we can calculate
\begin{align*}
\int_0^1 |\log \rho_1|^{q_1} \rho_1^{N_{k^{(1)}}-1} \,\mathrm{d}\rho_1
&= \int_{\infty}^0 |-u|^{q_1} (e^{-u})^{N_{k^{(1)}}-1} (-e^{-u}) \,\mathrm{d}u\\
&=\int_0^\infty u^{q_1} e^{-N_{k^{(1)}}u + u} e^{-u} \,\mathrm{d}u\\
&= \int_0^\infty u^{q_1} e^{-N_{k^{(1)}}u} \,\mathrm{d}u\\
&=\int_0^\infty \left( \frac{v}{N_{k^{(1)}}} \right)^{q_1} e^{-v} \frac{1}{N_{k^{(1)}}} \,\mathrm{d}v\\
&= \frac{1}{(N_{k^{(1)}})^{q_1+1}} \int_0^\infty v^{q_1} e^{-v} \,\mathrm{d}v\\
&= \frac{\Gamma(q_1+1)}{(N_{k^{(1)}})^{q_1+1}},
\end{align*}
where $\Gamma(z) = \int_0^\infty t^{z-1} e^{-t} \,\mathrm{d}t$.  We conclude that $b_0 \in \mathrm{CMO}_{\vec q, \vec{\tilde q}, k}(\mathbb{R}^n)$.

From a routine computation, there holds
\begin{align*}
\mathcal{H}_{\varphi}^{b_0}f_0(x)
&=\int_0^1\bigl(b_0(x)-b_0(tx)\bigr)f_0(tx)\varphi(t)\,\mathrm{d}t\\
&=\int_0^1\log\frac1t\prod_{i=1}^m(t\rho_i)^{-\frac{N_{k^{(i)}}}{\alpha_i}}\varphi(t)\,\mathrm{d}t\\
&=\prod_{i=1}^m\rho_i^{-\frac{N_{k^{(i)}}}{\alpha_i}}\int_0^1t^{-\sum\limits_{i=1}^m\frac{N_{k^{(i)}}}{\alpha_i}}\varphi(t)\log\frac1t\,\mathrm{d}t\\
&=f_0(x)\int_0^1t^{-\frac{N_k}{p}}\varphi(t)\log\frac1t\,\mathrm{d}t.
\end{align*}
Using the boundedness of $\mathcal{H}_{\varphi}^{b_0}$ from
$\dot{M}_{\vec r, \vec{\tilde r}, k}^p(\mathbb{R}^n)$ to
$\dot{M}_{\vec s, \vec{\tilde s}, k}^p(\mathbb{R}^n)$, and the fact that
$f_0\in\dot{M}_{\vec s, \vec{\tilde s}, k}^p(\mathbb{R}^n)$, we obtain
\[
\frac{\|f_0\|_{\dot{M}_{\vec s, \vec{\tilde s}, k}^p(\mathbb{R}^n)}}{\|f_0\|_{\dot{M}_{\vec r, \vec{\tilde r}, k}^p(\mathbb{R}^n)}}
\int_0^1t^{-\frac{N_k}{p}}\varphi(t)\log\frac1t\,\mathrm{d}t<\infty.
\]

That is to say,
\begin{equation*}
\int_0^1t^{-\frac{N_k}{p}}\varphi(t)\log\frac1t\,\mathrm{d}t<\infty.
\end{equation*}
We can also get
\begin{equation*}
\int_0^{\frac12}t^{-\frac{N_k}{p}}\varphi(t)\,\mathrm{d}t
\lesssim\int_0^{\frac12}t^{-\frac{N_k}{p}}\varphi(t)\log\frac1t\,\mathrm{d}t<\infty.
\end{equation*}
On the other hand, since $\varphi$ is integrable on $[\frac12,1]$ and $p>0$, we know that
\begin{equation*}
\int_{\frac12}^{1}t^{-\frac{N_k}{p}}\varphi(t)\,\mathrm{d}t\leq2^{\frac{N_k}{p}}\int_{\frac12}^{1}\varphi(t)\,\mathrm{d}t<\infty.
\end{equation*}

Therefore,
\begin{equation*}
\int_0^1t^{-\frac{N_k}{p}}\varphi(t)\,\mathrm{d}t<\infty.
\end{equation*}
Combining with
\[
\int_0^1t^{-\frac{N_k}{p}}\varphi(t)\log\frac1t\,\mathrm{d}t<\infty,
\]
we obtain
\[
\int_0^1t^{-\frac{N_k}{p}}\varphi(t)\log\frac{2}{t}\,\mathrm{d}t<\infty.
\]

Therefore,  the proof of Theorem \ref{3} is finished.
\end{proof}

\begin{proof}[Proof of Theorem 4.2]
Suppose that (\ref{12}) holds. Since $\varphi$ is locally integrable on $(0,1]$, condition (\ref{12}) implies
\[
\int_0^1 t^{N_k\lambda}\varphi(t)\,\mathrm{d}t<\infty.
\]
For $R>0$, by the definition of $\mathcal{H}_{\varphi,b}$, we have
\begin{align*}
&\frac{\left\|\mathcal{H}_{\varphi}^{b}f\,\chi_{B(0,R)}\right\|_{L_{\mathrm{rad}, k}^{\vec r} L_{\mathrm{ang}, k}^{\vec{\tilde r}}(\mathbb R^n)}
}{|B(0,R)|_k^\lambda\left\|\chi_{B(0,R)}\right\|_{L_{\mathrm{rad}, k}^{\vec r} L_{\mathrm{ang}, k}^{\vec{\tilde r}}(\mathbb R^n)}} \\
&=\frac{\left\|\int_0^1\bigl(b(\cdot)-b(t\cdot)\bigr)f(t\cdot)\varphi(t)\,\mathrm{d}t\,
\chi_{B(0,R)}(\cdot)\right\|_{L_{\mathrm{rad}, k}^{\vec r} L_{\mathrm{ang}, k}^{\vec{\tilde r}}(\mathbb R^n)}}{
|B(0,R)|_k^\lambda\left\|\chi_{B(0,R)}\right\|_{L_{\mathrm{rad}, k}^{\vec r} L_{\mathrm{ang}, k}^{\vec{\tilde r}}(\mathbb R^n)}} \\
&\leq J+JJ+JJJ,
\end{align*}
where
\begin{align*}
J&=\frac{\left\|\int_0^1\bigl(b(\cdot)-b_{B(0,R),k}\bigr)f(t\cdot)\varphi(t)\,\mathrm{d}t\,\chi_{B(0,R)}(\cdot)
\right\|_{L_{\mathrm{rad}, k}^{\vec r} L_{\mathrm{ang}, k}^{\vec{\tilde r}}(\mathbb R^n)}
}{|B(0,R)|_k^\lambda\left\|\chi_{B(0,R)}\right\|_{L_{\mathrm{rad}, k}^{\vec r} L_{\mathrm{ang}, k}^{\vec{\tilde r}}(\mathbb R^n)}},\\
JJ&=\frac{\left\|\int_0^1\bigl(b_{B(0,R),k}-b_{B(0,tR),k}\bigr)f(t\cdot)\varphi(t)\,\mathrm{d}t\,
\chi_{B(0,R)}(\cdot)\right\|_{L_{\mathrm{rad}, k}^{\vec r} L_{\mathrm{ang}, k}^{\vec{\tilde r}}(\mathbb R^n)}}{|B(0,R)|_k^\lambda
\left\|\chi_{B(0,R)}\right\|_{L_{\mathrm{rad}, k}^{\vec r} L_{\mathrm{ang}, k}^{\vec{\tilde r}}(\mathbb R^n)}},\\
JJJ&=\frac{\left\|\int_0^1\bigl(b_{B(0,tR),k}-b(t\cdot)\bigr)f(t\cdot)\varphi(t)\,\mathrm{d}t\,\chi_{B(0,R)}(\cdot)
\right\|_{L_{\mathrm{rad}, k}^{\vec r} L_{\mathrm{ang}, k}^{\vec{\tilde r}}(\mathbb R^n)}}{
|B(0,R)|_k^\lambda\left\|\chi_{B(0,R)}\right\|_{L_{\mathrm{rad}, k}^{\vec r} L_{\mathrm{ang}, k}^{\vec{\tilde r}}(\mathbb R^n)}}.
\end{align*}

For the first term $J$, using Minkowski's inequality and H\"older's inequality, we obtain
\begin{align*}
J&\leq\frac1{|B(0,R)|_k^\lambda\left\|\chi_{B(0,R)}\right\|_{L_{\mathrm{rad}, k}^{\vec r} L_{\mathrm{ang}, k}^{\vec{\tilde r}}(\mathbb R^n)}
}\int_0^1 \varphi(t) \\
&\quad\times\left\|\bigl(b(\cdot)-b_{B(0,R),k}\bigr)\chi_{B(0,R)}\right\|_{L_{\mathrm{rad}, k}^{\vec r_2} L_{\mathrm{ang}, k}^{\vec{\tilde r}_2}(\mathbb R^n)}
\left\|f(t\cdot)\chi_{B(0,R)}\right\|_{L_{\mathrm{rad}, k}^{\vec r_1} L_{\mathrm{ang}, k}^{\vec{\tilde r}_1}(\mathbb R^n)}\,\mathrm{d}t .
\end{align*}
By the definition of $\mathrm{CMO}_{\vec r_2, \vec{\tilde r}_2, k}(\mathbb R^n)$,
\[
\left\|
\bigl(b-b_{B(0,R),k}\bigr)\chi_{B(0,R)}
\right\|_{L_{\mathrm{rad}, k}^{\vec r_2} L_{\mathrm{ang}, k}^{\vec{\tilde r}_2}(\mathbb R^n)}
\leq
\left\|b\right\|_{\mathrm{CMO}_{\vec r_2, \vec{\tilde r}_2, k}(\mathbb R^n)}
\left\|\chi_{B(0,R)}\right\|_{L_{\mathrm{rad}, k}^{\vec r_2} L_{\mathrm{ang}, k}^{\vec{\tilde r}_2}(\mathbb R^n)}.
\]
Moreover, by the definition of $\dot{B}^{\vec r_1, \vec{\tilde r}_1,\lambda}_k(\mathbb R^n)$,
\begin{align*}
&\left\|f(t\cdot)\chi_{B(0,R)}\right\|_{L_{\mathrm{rad}, k}^{\vec r_1} L_{\mathrm{ang}, k}^{\vec{\tilde r}_1}(\mathbb R^n)}\\
&=t^{-\sum\limits_{i=1}^{m}\frac{N_{k^{(i)}}}{r_{1i}}}\left\|f\chi_{B(0,tR)}\right\|_{L_{\mathrm{rad}, k}^{\vec r_1} L_{\mathrm{ang}, k}^{\vec{\tilde r}_1}(\mathbb R^n)}\\
&\leq t^{-\sum\limits_{i=1}^{m}\frac{N_{k^{(i)}}}{r_{1i}}}\left\|f\right\|_{\dot{B}^{\vec r_1, \vec{\tilde r}_1,\lambda}_k} |B(0, tR)|_k^\lambda \left\|\chi_{B(0,tR)}\right\|_{L_{\mathrm{rad}, k}^{\vec r_1} L_{\mathrm{ang}, k}^{\vec{\tilde r}_1}(\mathbb R^n)}\\
&\leq t^{-\sum_{i=1}^{m}\frac{N_{k^{(i)}}}{r_{1i}}} \left\|f\right\|_{\dot{B}^{\vec r_1, \vec{\tilde r}_1,\lambda}_k} \left( t^{N_k\lambda}|B(0,R)|_k^\lambda \right) \left( t^{\sum_{i=1}^{m}\frac{N_{k^{(i)}}}{r_{1i}}}\left\|\chi_{B(0,R)}\right\|_{L_{\mathrm{rad}, k}^{\vec r_1} L_{\mathrm{ang}, k}^{\vec{\tilde r}_1}(\mathbb R^n)} \right) \\
&= t^{N_k\lambda} \left\|f\right\|_{\dot{B}^{\vec r_1, \vec{\tilde r}_1,\lambda}_k} |B(0,R)|_k^\lambda \left\|\chi_{B(0,R)}\right\|_{L_{\mathrm{rad}, k}^{\vec r_1} L_{\mathrm{ang}, k}^{\vec{\tilde r}_1}(\mathbb R^n)}.
\end{align*}
We get
\begin{align*}
J &\leq \int_0^1 \varphi(t) \left\|b\right\|_{\mathrm{CMO}_{\vec r_2, \vec{\tilde r}_2, k}} t^{N_k\lambda} \left\|f\right\|_{\dot{B}^{\vec r_1, \vec{\tilde r}_1,\lambda}_k} \frac{ |B(0,R)|_k^\lambda \left\|\chi_{B(0,R)}\right\|_{L_{\mathrm{rad}, k}^{\vec r_1} L_{\mathrm{ang}, k}^{\vec{\tilde r}_1}} \left\|\chi_{B(0,R)}\right\|_{L_{\mathrm{rad}, k}^{\vec r_2} L_{\mathrm{ang}, k}^{\vec{\tilde r}_2}} }{ |B(0,R)|_k^\lambda \left\|\chi_{B(0,R)}\right\|_{L_{\mathrm{rad}, k}^{\vec r} L_{\mathrm{ang}, k}^{\vec{\tilde r}}} } \,\mathrm{d}t \\
&= \left\|b\right\|_{\mathrm{CMO}_{\vec r_2, \vec{\tilde r}_2, k}} \left\|f\right\|_{\dot{B}^{\vec r_1, \vec{\tilde r}_1,\lambda}_k} \left( \int_0^1 t^{N_k\lambda} \varphi(t) \,\mathrm{d}t \right) \frac{ \left\|\chi_{B(0,R)}\right\|_{L_{\mathrm{rad}, k}^{\vec r_1} L_{\mathrm{ang}, k}^{\vec{\tilde r}_1}} \left\|\chi_{B(0,R)}\right\|_{L_{\mathrm{rad}, k}^{\vec r_2} L_{\mathrm{ang}, k}^{\vec{\tilde r}_2}} }{ \left\|\chi_{B(0,R)}\right\|_{L_{\mathrm{rad}, k}^{\vec r} L_{\mathrm{ang}, k}^{\vec{\tilde r}}} }\\
&\lesssim \left\|b\right\|_{\mathrm{CMO}_{\vec r_2, \vec{\tilde r}_2, k}(\mathbb R^n)} \left\|f\right\|_{\dot{B}^{\vec r_1, \vec{\tilde r}_1,\lambda}_k(\mathbb R^n)} \int_0^1 t^{N_k\lambda}\varphi(t)\,\mathrm{d}t.
\end{align*}

For the last term $JJJ$, by Minkowski's inequality and H\"older's inequality, we have
\begin{align*}
JJJ &\leq \frac{1}{|B(0, R)|_k^{\lambda} \|\chi_{B(0, R)}\|_{L_{\mathrm{rad}, k}^{\vec{r}} L_{\mathrm{ang}, k}^{\vec{\tilde{r}}}(\mathbb R^n)}} \int_0^1 \varphi(t) \\
&\quad\times \|f(t\cdot)\chi_{B(0, R)}\|_{L_{\mathrm{rad}, k}^{\vec{r}_1} L_{\mathrm{ang}, k}^{\vec{\tilde{r}}_1}(\mathbb R^n)}
\|(b_{B(0, tR), k} - b(t\cdot))\chi_{B(0, R)}\|_{L_{\mathrm{rad}, k}^{\vec{r}_2} L_{\mathrm{ang}, k}^{\vec{\tilde{r}}_2}(\mathbb R^n)} \,\mathrm{d}t.
\end{align*}

Because of
\begin{equation*}
\|\chi_{B(0, tR)}\|_{L_{\mathrm{rad}, k}^{\vec{r}_2} L_{\mathrm{ang}, k}^{\vec{\tilde{r}}_2}(\mathbb R^n)} = t^{\sum\limits_{i=1}^m \frac{N_{k^{(i)}}}{r_{2i}}} \|\chi_{B(0, R)}\|_{L_{\mathrm{rad}, k}^{\vec{r}_2} L_{\mathrm{ang}, k}^{\vec{\tilde{r}}_2}(\mathbb R^n)},
\end{equation*}
we can get
\begin{align*}
&\left\|(b_{B(0, tR), k} - b(t\cdot))\chi_{B(0, R)}\right\|_{L_{\mathrm{rad}, k}^{\vec{r}_2} L_{\mathrm{ang}, k}^{\vec{\tilde{r}}_2}(\mathbb R^n)} \\
&= t^{-\sum\limits_{i=1}^m \frac{N_{k^{(i)}}}{r_{2i}}} \left\|(b_{B(0, tR), k} - b)\chi_{B(0, tR)}\right\|_{L_{\mathrm{rad}, k}^{\vec{r}_2} L_{\mathrm{ang}, k}^{\vec{\tilde{r}}_2}(\mathbb R^n)} \\
&\leq t^{-\sum\limits_{i=1}^m \frac{N_{k^{(i)}}}{r_{2i}}} \|b\|_{\mathrm{CMO}_{\vec{r}_2, \vec{\tilde{r}}_2, k}(\mathbb R^n)} \|\chi_{B(0, tR)}\|_{L_{\mathrm{rad}, k}^{\vec{r}_2} L_{\mathrm{ang}, k}^{\vec{\tilde{r}}_2}(\mathbb R^n)}\\
&= \|b\|_{\mathrm{CMO}_{\vec{r}_2, \vec{\tilde{r}}_2, k}(\mathbb R^n)} \|\chi_{B(0, R)}\|_{L_{\mathrm{rad}, k}^{\vec{r}_2} L_{\mathrm{ang}, k}^{\vec{\tilde{r}}_2}(\mathbb R^n)}.
\end{align*}
Along with the previous estimate for $f(t\cdot)$, we obtain
\begin{align*}
JJJ
&\lesssim\left\|b\right\|_{\mathrm{CMO}_{\vec r_2, \vec{\tilde r}_2, k}(\mathbb R^n)}
\left\|f\right\|_{\dot{B}^{\vec r_1, \vec{\tilde r}_1,\lambda}_k(\mathbb R^n)}
\int_0^1 t^{N_k\lambda}\varphi(t)\,\mathrm{d}t .
\end{align*}

Finally, we estimate the middle term $JJ$. By Minkowski's inequality and H\"older's inequality,
\begin{align*}
JJ
&\leq\frac1{|B(0,R)|_k^\lambda\left\|\chi_{B(0,R)}\right\|_{L_{\mathrm{rad}, k}^{\vec r} L_{\mathrm{ang}, k}^{\vec{\tilde r}}(\mathbb R^n)}}
\int_0^1\left|b_{B(0,R),k}-b_{B(0,tR),k}\right|\varphi(t) \\
&\quad\times\left\|f(t\cdot)\chi_{B(0,R)}\right\|_{L_{\mathrm{rad}, k}^{\vec r_1} L_{\mathrm{ang}, k}^{\vec{\tilde r}_1}(\mathbb R^n)}
\left\|\chi_{B(0,R)}\right\|_{L_{\mathrm{rad}, k}^{\vec r_2} L_{\mathrm{ang}, k}^{\vec{\tilde r}_2}(\mathbb R^n)}\,\mathrm{d}t \\
&\lesssim\left\|f\right\|_{\dot{B}^{\vec r_1, \vec{\tilde r}_1,\lambda}_k(\mathbb R^n)}\int_0^1\left|b_{B(0,R),k}-b_{B(0,tR),k}\right|t^{N_k\lambda}\varphi(t)\,\mathrm{d}t .
\end{align*}

Similarly to Theorem \ref{3}, we have
\begin{align*}
|b_{B(0,R),k} - b_{B(0,tR),k}|
&\leq \sum_{v=1}^{l} |b_{B(0,2^{-v+1}R),k} - b_{B(0,2^{-v}R),k}| + |b_{B(0,2^{-l}R),k} - b_{B(0,tR),k}|\\
&\lesssim \sum_{v=1}^{l} \|b\|_{\mathrm{CMO}_{\vec r_2, \vec{\tilde r}_2, k}(\mathbb R^n)} +  \|b\|_{\mathrm{CMO}_{\vec r_2, \vec{\tilde r}_2, k}(\mathbb R^n)} \\
&= (l+1) \|b\|_{\mathrm{CMO}_{\vec r_2, \vec{\tilde r}_2, k}(\mathbb R^n)}.
\end{align*}

Finally, the relation $ t< 2^{-l+1}$ implies $l-1 < \log_2(1/t)$. Thus, $l+1 \lesssim 1 + \log(1/t)$, which leads to our desired estimate:
\begin{align*}
|b_{B(0,R),k} - b_{B(0,tR),k}| \lesssim \|b\|_{\mathrm{CMO}_{\vec r_2, \vec{\tilde r}_2, k}(\mathbb R^n)} \left(1 + \log\frac{1}{t}\right).
\end{align*}

Therefore,
\[
JJ\lesssim
\left\|f\right\|_{\dot{B}^{\vec r_1, \vec{\tilde r}_1,\lambda}_k(\mathbb R^n)}
\left\|b\right\|_{\mathrm{CMO}_{\vec r_2, \vec{\tilde r}_2, k}(\mathbb R^n)}
\int_0^1t^{N_k\lambda}\left(1+\log\frac1t\right)\varphi(t)\,\mathrm{d}t .
\]
Combining the estimates of $J$, $JJ$ and $JJJ$, we obtain
\[
\left\|\mathcal{H}_{\varphi}^{b}f\right\|_{\dot{B}^{\vec r, \vec{\tilde r},\lambda}_k(\mathbb R^n)}
\lesssim\left\|b\right\|_{\mathrm{CMO}_{\vec r_2, \vec{\tilde r}_2, k}(\mathbb R^n)}
\left\|f\right\|_{\dot{B}^{\vec r_1, \vec{\tilde r}_1,\lambda}_k(\mathbb R^n)}
\int_0^1t^{N_k\lambda}\left(1+\log\frac1t\right)\varphi(t)\,\mathrm{d}t .
\]
Since (\ref{12}) and the local integrability of $\varphi$ imply
\[
\int_0^1t^{N_k\lambda}\left(1+\log\frac1t\right)\varphi(t)\,\mathrm{d}t<\infty,
\]
it follows that $\mathcal{H}_{\varphi}^{b}$ is bounded from
$\dot{B}^{\vec r_1, \vec{\tilde r}_1,\lambda}_k(\mathbb R^n)$ to
$\dot{B}^{\vec r, \vec{\tilde r},\lambda}_k(\mathbb R^n)$.

Conversely, suppose that $\mathcal{H}_{\varphi}^{b}$ is bounded from
$\dot{B}^{\vec r_1, \vec{\tilde r}_1,\lambda}_k(\mathbb R^n)$ to
$\dot{B}^{\vec r, \vec{\tilde r},\lambda}_k(\mathbb R^n)$ for all
$b\in\mathrm{CMO}_{\vec r_2, \vec{\tilde r}_2, k}(\mathbb R^n)$.
Let
\[
b_0(x)=\log |x|,
\qquad x\neq0,
\]
and define $b_0(0)=0$. Similarly to Theorem \ref{3}, it is known that
\[
b_0\in\mathrm{CMO}_{\vec r_2, \vec{\tilde r}_2, k}(\mathbb R^n).
\]
Since $\lambda>-\frac1{N_k}\sum\limits_{i=1}^{m}\frac{N_{k^{(i)}}}{r_{1i}},$
we can choose real numbers $\lambda_1,\ldots,\lambda_m$ such that
$\lambda_i>-\frac{N_{k^{(i)}}}{r_{1i}},\, i=1,\ldots,m,$
and $\sum\limits_{i=1}^{m}\lambda_i=N_k\lambda.$
Let us define a purely radial function
\[
f_0(x) = f_0(\rho_1\omega_1, \ldots, \rho_m\omega_m) = \prod_{i=1}^{m} \rho_i^{\lambda_i}.
\]
Similarly to Theorem \ref{2}, then
\[f_0\in \dot{B}^{\vec r_1, \vec{\tilde r}_1,\lambda}_k(\mathbb R^n)\cap
\dot{B}^{\vec r, \vec{\tilde r},\lambda}_k(\mathbb R^n).
\]

For $x\neq0$, we have
\begin{align*}
\mathcal{H}_{\varphi}^{b_0} f_0(x)
&=\int_0^1\bigl(\log|x|-\log|tx|\bigr)\prod_{i=1}^{m}(t\rho_i)^{\lambda_i}\varphi(t)\,\mathrm{d}t \\
&=f_0(x)\int_0^1t^{\sum_{i=1}^{m}\lambda_i}\log\frac1t\,\varphi(t)\,\mathrm{d}t \\
&=f_0(x)\int_0^1t^{N_k\lambda}\log\frac1t\,\varphi(t)\,\mathrm{d}t .
\end{align*}
By the boundedness of $\mathcal{H}_{\varphi}^{b_0}$, we get
\[
\left\|\mathcal{H}_{\varphi}^{b_0}f_0\right\|_{\dot{B}^{\vec r, \vec{\tilde r},\lambda}_k(\mathbb R^n)}
\leq\left(\int_0^1t^{N_k\lambda}\log\frac1t\,\varphi(t)\,\mathrm{d}t\right)
\left\|f_0\right\|_{\dot{B}^{\vec r, \vec{\tilde r},\lambda}_k(\mathbb R^n)}<\infty.
\]
Since $f_0\not\equiv0$, it follows that
\[
\int_0^1
t^{N_k\lambda}
\log\frac1t\,\varphi(t)\,\mathrm{d}t<\infty.
\]
This completes the proof of Theorem 4.2.
\end{proof}

\begin{proof}[Proof of Theorem 4.3]
Suppose that (\ref{13}) holds. For $R>0$, as in the proof of Theorem \ref{4}, we decompose
\[
\frac{\left\|\mathcal{H}_{\varphi}^{b}f\,\chi_{B(0,R)}\right\|_{L_{\mathrm{rad}, k}^{\vec r} L_{\mathrm{ang}, k}^{\vec{\tilde r}}(\mathbb R^n)}
}{|B(0,R)|_k^\lambda\left\|\chi_{B(0,R)}\right\|_{L_{\mathrm{rad}, k}^{\vec r} L_{\mathrm{ang}, k}^{\vec{\tilde r}}(\mathbb R^n)}}\leq K+KK+KKK,
\]
where
\begin{align*}
K&=\frac{\left\|\int_0^1\bigl(b(\cdot)-b_{B(0,R),k}\bigr)f(t\cdot)\varphi(t)\,\mathrm{d}t\,\chi_{B(0,R)}(\cdot)
\right\|_{L_{\mathrm{rad}, k}^{\vec r} L_{\mathrm{ang}, k}^{\vec{\tilde r}}(\mathbb R^n)}
}{|B(0,R)|_k^\lambda\left\|\chi_{B(0,R)}\right\|_{L_{\mathrm{rad}, k}^{\vec r} L_{\mathrm{ang}, k}^{\vec{\tilde r}}(\mathbb R^n)}},\\
KK&=\frac{\left\|\int_0^1\bigl(b_{B(0,R),k}-b_{B(0,tR),k}\bigr)f(t\cdot)\varphi(t)\,\mathrm{d}t\,
\chi_{B(0,R)}(\cdot)\right\|_{L_{\mathrm{rad}, k}^{\vec r} L_{\mathrm{ang}, k}^{\vec{\tilde r}}(\mathbb R^n)}}{|B(0,R)|_k^\lambda
\left\|\chi_{B(0,R)}\right\|_{L_{\mathrm{rad}, k}^{\vec r} L_{\mathrm{ang}, k}^{\vec{\tilde r}}(\mathbb R^n)}},\\
KKK&=\frac{\left\|\int_0^1\bigl(b_{B(0,tR),k}-b(t\cdot)\bigr)f(t\cdot)\varphi(t)\,\mathrm{d}t\,\chi_{B(0,R)}(\cdot)
\right\|_{L_{\mathrm{rad}, k}^{\vec r} L_{\mathrm{ang}, k}^{\vec{\tilde r}}(\mathbb R^n)}}{
|B(0,R)|_k^\lambda\left\|\chi_{B(0,R)}\right\|_{L_{\mathrm{rad}, k}^{\vec r} L_{\mathrm{ang}, k}^{\vec{\tilde r}}(\mathbb R^n)}}.
\end{align*}

Similarly to Theorem \ref{4}, we obtain
\[
K\lesssim
\left\|b\right\|_{\mathrm{CMO}_{\vec r_2, \vec{\tilde r}_2,\lambda_2,k}(\mathbb R^n)}
\left\|f\right\|_{\dot{B}^{\vec r_1, \vec{\tilde r}_1,\lambda_1}_k(\mathbb R^n)}
\int_0^1 t^{N_k\lambda_1}\varphi(t)\,\mathrm{d}t,
\]
and
\begin{align*}
KKK\lesssim
\left\|b\right\|_{\mathrm{CMO}_{\vec r_2, \vec{\tilde r}_2,\lambda_2,k}(\mathbb R^n)}
\left\|f\right\|_{\dot{B}^{\vec r_1, \vec{\tilde r}_1,\lambda_1}_k(\mathbb R^n)}
\int_0^1 t^{N_k\lambda_1}\varphi(t)\,\mathrm{d}t.
\end{align*}

Finally, for $KK$, we have
\begin{align*}
KK&\lesssim
\left\|f\right\|_{\dot{B}^{\vec r_1, \vec{\tilde r}_1,\lambda_1}_k(\mathbb R^n)}
\int_0^1\left|b_{B(0,R),k}-b_{B(0,tR),k}\right|t^{N_k\lambda_1}\varphi(t)\,\frac{\mathrm{d}t}{|B(0,R)|_k^{\lambda_2}} .
\end{align*}

Let $0<t<1$, and choose $l\in\mathbb N\setminus\{0\}$ such that
\[
2^{-l}<t\leq2^{-l+1}.
\]
Then
\begin{align*}
\left|b_{B(0,R),k}-b_{B(0,tR),k}\right|
\leq\sum_{v=1}^{l}\left|b_{B(0,2^{-v+1}R),k}-b_{B(0,2^{-v}R),k}\right|+\left|b_{B(0,2^{-l}R),k}-b_{B(0,tR),k}\right|.
\end{align*}
Let $(\vec r_2)'=(r_{21}',\ldots,r_{2m}')$ and
$(\vec{\tilde r}_2)'=(\tilde r_{21}',\ldots,\tilde r_{2m}')$ be defined by
\[
\frac{1}{r_{2i}}+\frac{1}{r_{2i}'}=1,
\qquad\frac{1}{\tilde r_{2i}}+\frac{1}{\tilde r_{2i}'}=1,\qquad i=1,\ldots,m.
\]
For every $v=1,\ldots,l$, using
\[
|B(0,2^{-v+1}R)|_k
=
2^{N_k}|B(0,2^{-v}R)|_k,
\]
the mixed H\"older inequality, and the definition of
$\mathrm{CMO}_{\vec r_2,\vec{\tilde r}_2,\lambda_2,k}(\mathbb R^n)$, we obtain
\begin{align*}
&\left|b_{B(0,2^{-v}R),k}-b_{B(0,2^{-v+1}R),k}\right|\\
&=\frac{1}{|B(0,2^{-v}R)|_k}\left|\int_{B(0,2^{-v}R)}\left(b(x)-b_{B(0,2^{-v+1}R),k}\right)\,\mathrm{d}\mu_k(x)\right|\\
&\lesssim\frac{1}{|B(0,2^{-v+1}R)|_k}\int_{B(0,2^{-v+1}R)}\left|b(x)-b_{B(0,2^{-v+1}R),k}\right|\,\mathrm{d}\mu_k(x)\\
&\leq\frac{1}{|B(0,2^{-v+1}R)|_k}\left\|\left(b-b_{B(0,2^{-v+1}R),k}\right)\chi_{B(0,2^{-v+1}R)}\right\|_{L_{\mathrm{rad},k}^{\vec r_2}L_{\mathrm{ang},k}^{\vec{\tilde r}_2}(\mathbb R^n)}\\
&\quad\times\left\|\chi_{B(0,2^{-v+1}R)}\right\|_{L_{\mathrm{rad},k}^{(\vec r_2)'}L_{\mathrm{ang},k}^{(\vec{\tilde r}_2)'}(\mathbb R^n)}\\
&\lesssim\left\|b\right\|_{\mathrm{CMO}_{\vec r_2,\vec{\tilde r}_2,\lambda_2,k}(\mathbb R^n)}|B(0,2^{-v+1}R)|_k^{\lambda_2}\\
&=2^{-N_k\lambda_2(v-1)}\left\|b\right\|_{\mathrm{CMO}_{\vec r_2,\vec{\tilde r}_2,\lambda_2,k}(\mathbb R^n)}|B(0,R)|_k^{\lambda_2},
\end{align*}
where we have used
\[
\left\|\chi_{B(0,\rho)}\right\|_{L_{\mathrm{rad},k}^{\vec r_2}L_{\mathrm{ang},k}^{\vec{\tilde r}_2}(\mathbb R^n)}
\left\|\chi_{B(0,\rho)}\right\|_{L_{\mathrm{rad},k}^{(\vec r_2)'}L_{\mathrm{ang},k}^{(\vec{\tilde r}_2)'}(\mathbb R^n)}
\lesssim|B(0,\rho)|_k.
\]
Since $\lambda_2>0$, it follows that
\begin{align*}
&\sum_{v=1}^{l}
\left|b_{B(0,2^{-v+1}R),k}-b_{B(0,2^{-v}R),k}\right|\\
&\lesssim\left\|b\right\|_{\mathrm{CMO}_{\vec r_2,\vec{\tilde r}_2,\lambda_2,k}(\mathbb R^n)}|B(0,R)|_k^{\lambda_2}\sum_{v=1}^{l}2^{-N_k\lambda_2(v-1)}\\
&\lesssim\left\|b\right\|_{\mathrm{CMO}_{\vec r_2,\vec{\tilde r}_2,\lambda_2,k}(\mathbb R^n)}|B(0,R)|_k^{\lambda_2}.
\end{align*}
Moreover, since
\[
B(0,2^{-l}R)\subset B(0,tR)
\quad\text{and}\quad
|B(0,tR)|_k\leq2^{N_k}|B(0,2^{-l}R)|_k,
\]
we similarly obtain
\begin{align*}
&\left|b_{B(0,2^{-l}R),k}-b_{B(0,tR),k}\right|\\
&=\frac{1}{|B(0,2^{-l}R)|_k}\left|\int_{B(0,2^{-l}R)}\left(b(x)-b_{B(0,tR),k}\right)\,\mathrm{d}\mu_k(x)\right|\\
&\lesssim\frac{1}{|B(0,tR)|_k}\int_{B(0,tR)}\left|b(x)-b_{B(0,tR),k}\right|\,\mathrm{d}\mu_k(x)\\
&\lesssim\left\|b\right\|_{\mathrm{CMO}_{\vec r_2,\vec{\tilde r}_2,\lambda_2,k}(\mathbb R^n)}|B(0,tR)|_k^{\lambda_2}\\
&\leq\left\|b\right\|_{\mathrm{CMO}_{\vec r_2,\vec{\tilde r}_2,\lambda_2,k}(\mathbb R^n)}
|B(0,R)|_k^{\lambda_2}.
\end{align*}

Therefore,
\[
KK\lesssim
\left\|b\right\|_{\mathrm{CMO}_{\vec r_2, \vec{\tilde r}_2,\lambda_2,k}(\mathbb R^n)}\left\|f\right\|_{\dot{B}^{\vec r_1, \vec{\tilde r}_1,\lambda_1}_k(\mathbb R^n)}
\int_0^1 t^{N_k\lambda_1}\varphi(t)\,\mathrm{d}t .
\]
Combining the estimates of $K$, $KK$ and $KKK$, we obtain
\[
\left\|\mathcal{H}_{\varphi}^{b}f\right\|_{\dot{B}^{\vec r, \vec{\tilde r},\lambda}_k(\mathbb R^n)}
\lesssim\left\|b\right\|_{\mathrm{CMO}_{\vec r_2, \vec{\tilde r}_2,\lambda_2,k}(\mathbb R^n)}
\left\|f\right\|_{\dot{B}^{\vec r_1, \vec{\tilde r}_1,\lambda_1}_k(\mathbb R^n)}
\int_0^1 t^{N_k\lambda_1}\varphi(t)\,\mathrm{d}t .
\]
Hence $\mathcal{H}_{\varphi}^{b}$ is bounded from
$\dot{B}^{\vec r_1, \vec{\tilde r}_1,\lambda_1}_k(\mathbb R^n)$ to
$\dot{B}^{\vec r, \vec{\tilde r},\lambda}_k(\mathbb R^n)$.
This finishes the proof  of Theorem \ref{5}.
\end{proof}

\end{document}